\documentclass[]{interact}
\usepackage{graphicx,color}
\usepackage{xcolor}
\usepackage{epstopdf}
\usepackage[caption=false]{subfig}

\usepackage[numbers,sort&compress]{natbib}
\bibpunct[, ]{[}{]}{,}{n}{,}{,}
\makeatletter
\def\NAT@def@citea{\def\@citea{\NAT@separator}}
\makeatother

\theoremstyle{plain}
\newtheorem{theorem}{Theorem}[section]
\newtheorem{lemma}[theorem]{Lemma}
\newtheorem{corollary}[theorem]{Corollary}

\theoremstyle{definition}
\newtheorem{definition}[theorem]{Definition}

\theoremstyle{remark}
\newtheorem{remark}{Remark}

\begin{document}


\title{Lipschitz stability  estimate  and reconstruction  of Lam\'e  parameters in linear  elasticity}
\author{
\name{S. Eberle,\textsuperscript{a}
\thanks{CONTACT  S.  Eberle. Email: eberle@math.uni-frankfurt.de} 
B.  Harrach,\textsuperscript{b}
\thanks{CONTACT  B.  Harrach. Email: harrach@math.uni-frankfurt.de} 
H. Meftahi,\textsuperscript{c}
\thanks{CONTACT  H.  Meftahi. Email: houcine.meftahi@enit.utm.tn} and 
T.  Rezgui\textsuperscript{d}
\thanks{CONTACT  T.  Rezgui. Email: taher.rezgui@enit.utm.tn}
}
\affil{\textsuperscript{a,b}Institute for Mathematics, Goethe-University Frankfurt, 60325 Frankfurt am Main, Germany; 
\textsuperscript{c,d}ENIT of Tunisia, Tunis, 1002 Tunisia.}
}
\maketitle
\begin{abstract}
In this paper, we consider the  inverse problem   of recovering an isotropic  elastic tensor  from the Neumann-to-Dirichlet 
map. {\color{black} To this end, we prove a Lipschitz stability estimate  for Lam\'e  parameters with certain regularity assumptions. In addition, we assume that the Lam\' e parameters belong to a known finite subspace with a priori known  bounds and that they fulfill a monotonicity property.}
The proof relies on a monotonicity  result combined with the techniques of localized potentials.  
To numerically solve the inverse problem, we propose a Kohn-Vogelius-type cost functional over a class  of admissible  
parameters  subject to two boundary value problems. 
 {\color{black} The reformulation of the minimization  problem via the Neumann-to-Dirichlet operator
allows us to obtain the optimality conditions by using the Fr\'echet differentiability of this operator and its inverse}. 
The reconstruction is then performed by means of an iterative algorithm based on a  quasi-Newton method. Finally, 
we give and discuss several numerical examples. 
\end{abstract}
\section{Introduction}
In this paper,  we consider the inverse problem of recovering the elastic tensor $\mathbb{C}$ of a linear  isotropic elastic  body from the Neumann-to-Dirichlet 
operator $\Lambda(\mathbb{C})$. 
{\color{black}The main motivations of this problem are non-destructive testing for elastic bodies 
in order to detect and reconstruct material inclusions (as presented in \cite{Ammari} and the references therein), geophysical (see, e.g., \cite{Weglein}), and medical applications (as considered in \cite{Barbone_Gokhale}), in particular localization of potential tumors via a 
medical imaging modality called elastography.} Elastography is concerned with the reconstruction of the elastic properties in biological  
tissues and the present article aims at giving access to these features. 
{\color{black}For a topical review of inverse problems in elasticity and their applications
the reader is, e.g., referred to \cite{Bonnet}.}
\\
{\color{black}
The paper is split into two parts and gives a twofold perspective on determining Lam\'e parameters in linear elasticity. Part
one is on proving a Lipschitz stability
estimate based on the Neumann-to-Dirichlet map for
the corresponding tensors $\mathbb{C}_1$ and $\mathbb{C}_2$, depending on the Lam\'e parameters 
$(\lambda_1,\mu_1)$ and $(\lambda_2,\mu_2)$,
respectively. The estimate is given under the following a priori assumptions.
\begin{itemize}
\item[(i)] The Lam\'e parameters $(\lambda_j,\mu_j)$ for $j = 1, 2$ belong to a known finite dimensional subspace, with $\lambda$ being piecewise continuous and $\mu$ being Lipschitz continuous, and with fixed uniform lower and upper bounds.
\item[(ii)] The two pairs of parameters satisfy a monotonicity assumption that $(\lambda_1,\mu_1)$ is, as a pair, either a lower or an upper bound to $(\lambda_2,\mu_2)$.
\end{itemize}
\noindent
Part two deals with the reconstruction of the parameters $(\lambda,\mu)$ based on minimizing a Kohn-Vogelius type
functional.
Let us stress that this numerical part does not build on the theoretical results presented in Section 3 but rather approaches the problem
from a heuristic numerical side to demonstrate that useful numerical reconstructions are indeed possible. It remains a challenging open task 
how to unite the theoretical and numerical approaches in order to find rigorously justified reconstruction methods that work well in
practically relevant settings.}
\\
From the theoretical point of view, the inverse problem  of recovering $\mathbb{C}$ (or the Lam\'e moduli $\lambda, \mu$; cf. \eqref{tensor}) 
has been studied  by several authors. In the two dimensional case Ikehata \cite{ikehata1990inversion}  proves that the deflection $h$ between $(\lambda+h,\mu+h)$  and $(\lambda,\mu)$ can be uniquely determined by the first-order approximation of the Dirichlet-to-Neumann operator. Akamatsu,  Nakamura and Steinberg \cite{akamatsu1991identification}, give an inversion formula   for the normal derivatives at the boundary of the Lam\'e coefficients $\lambda, \mu\in C^\infty$ from the   Dirichlet-to-Neumann map. At the same time they present stability estimates for the boundary values of $\lambda, \mu$.
  Nakamura and Uhlmann \cite{nakamura1993identification} established that the Lam\'e coefficients are uniquely determined from the Dirichlet-to-Neumann operator,  assuming that they are sufficiently close to a pair of positive constants.  Imanuvilov and Yamamoto \cite{imanuvilov2011reconstruction} proved  that the Lam\'e coefficient  $\lambda$ can be recovered from partial Cauchy data if the coefficient $\mu$ is some positive constant. 
{\color{black} Boundary determination of Lam\'e coefficients can be found in the work by Lin and Nakamura
\cite{lin2017boundary}.}
 {\color{black} In addition, we want to mention a global uniqueness results in 2D by Imanuvilov and Yamamoto \cite{Imanuvilov2015}}.
 \\
  For the three dimensional case, Nakamura and Uhlmann  \cite{nakamura1995inverse, nakamura2003global} and Eskin and Ralston \cite{eskin2002inverse} proved uniqueness results  for both Lam\'e coefficients when   $\mu$ is assumed to be  close to a positive constant. The proofs in the above papers rely on the  construction of complex geometric optics solutions.  
For a partial data version, uniqueness for recovering piecewise constant Lam\'e parameters {\color{black} 
and a quantitative Lipschitz stability result} was proved in \cite{beretta2014lipschitz, beretta2014uniqueness}, and some boundary determination results
were shown in \cite{nakamura1999layer,nakamura1995inverse,lin2017boundary}. For fully anisotropic $\mathbb{C}$, 
uniqueness was proved {\color{black} by C$\hat{\mathrm a}$rstea, Honda and Nakamura \cite{carstea2018uniqueness}} for a piecewise homogeneous medium.
 Isakov,  Wang and  Yamamoto \cite{isakov2007inverse}, proved  
 H\"older and {\mbox Lipschitz} stability estimates of determining all coefficients of a dynamical Lam\'e  system with residual stress, including the density Lam\'e parameters,  
 and the residual stress, by three pairs of observations from the whole boundary or from a part of it.
 \\
  {\color{black}In this paper, we prove a Lipschitz stability result when the Lam\' e  coefficient   $\lambda$  is piecewise continuous, $\mu$ is Lipschitz and a monotonicity assumption holds.  
In more detail, we assume that the Lam\' e parameters belong to a known finite subspace with a priori known lower and upper bounds.}
 Our approach relies on the monotonicity
 of the Neumann-to-Dirichlet operator with respect to the elastic tensor and the techniques of localized potentials 
 {\color{black}\cite{harrach2018helmholtz,harrach2018localizing,harrach2012simultaneous,harrach2018fractional,
harrach2018fractional2,gebauer2008localized,harrach2009uniqueness}}.   For the numerical solution, we reformulate the inverse problem into a minimization problem using a Kohn-Vogelius
 type cost functional, and use a quasi-Newton method which employs the analytic gradient   of  the   cost function  and the approximation of the inverse Hessian is updated by  a BFGS (Broyden, Fletcher, Goldfarb, Shanno) scheme \cite{murea2005bfgs}.
\\
 Let us give some more remarks on the relation of this work to previous results.  
 {\color{black} Stability for inverse coefficient problems are derived in general from technically challenging approaches involving quantitative unique continuation estimates and the study of special solutions} \cite{bellassoued2008inverse,bellassoued2007lipschitz,nakamura1995inverse,nakamura2003global}.
 Our approach on proving a Lipschitz stability result  is relatively simple and easy to extend to other settings, and has already
led to new results on uniqueness and Lipschitz stability in electrical impedance tomography (EIT) with finitely many electrodes \cite{harrach2019uniqueness} as well as for the inverse Robin transmission problem
\cite{harrach2018global,harrach2020uniqueness} and on the stability in machine learning
reconstruction algorithms \cite{seo2018learning} under a definiteness assumption.
\\
 The paper is organized as follows. In Section \ref{sec2}, we introduce the forward as well as  the inverse problem
 and the Neumann-to-Dirichlet operator. In Section \ref{sec3}, we show a monotonicity result between the Lam\'e parameters 
 and the Neumann-to-Dirichlet operator and deduce the existence of localized potentials. Then, we prove the
 Lipschitz stability
   estimate {\color{black} for a finite dimensional subset with bounded Lam\'e parameters and a monotonicity 
   assumption. In Section \ref{sec4}, we present a numerical approach for solving the inverse problem and
   in Section \ref{sec5} we show some numerical results.}

\section{Problem formulation}\label{sec2}
Let $\Omega \subset \mathbb{R}^d$ $(d \geq 2)$, be a bounded and connected open set, occupied by an isotropic material with  linear stress-strain relation. The boundary $\partial \Omega$,   is assumed to be $C^{1,1}$ and  consists of two {non-empty} disjoint {open}
parts, the fixed "Dirichlet-boundary" $\Gamma_{\textup D}$ and the "Neumann-boundary" $\Gamma_{\textup N}$ {with}
\[
{\partial \Omega=\overline{\Gamma_{\textup N}}\cup \overline{\Gamma_{\textup D}}},\quad  \Gamma_{\textup N}\cap \Gamma_{\textup D}=\emptyset.
\]
{\color{black}The choice of mixed boundary conditions is based on the physical treatment of the elasticity problem.
The Neumann-to-Dirichlet operator with fixed Dirichlet part is an idealized model for fixing an elastic object in place on one part of the boundary, applying different pressure patterns
to the remaining part and measuring the resulting displacements.}
\\
\\
We denote a given surface load by $g\in L^{2}(\Gamma_{\textup N})^d$. Then the displacement vector 
$u: \overline{\Omega}\rightarrow \mathbb{R}^d$ satisfies the following boundary value problem
 \begin{equation}
\label{direct}
\left\{
\begin{aligned}
-\mathrm{div}(\mathbb{C}\hat \nabla u) & = 0&\quad &\text{ in }\Omega,\\
 (\mathbb{C}\hat\nabla u)\nu&= g&\quad &\text{ on  }\Gamma_{\textup N},\\
  u&= 0&\quad &\text{ on  }\Gamma_{\textup D},
\end{aligned}
\right.
\end{equation}
where  $\nu$ is the outer unit normal vector to  $\partial\Omega$, {\color{black} which is similar to the boundary value problem considered, e.g., in \cite{Ciarlet}.} The  linearized  strain  tensor $\hat \nabla u$  and the stress tensor  $\mathbb{C}\hat\nabla u$ are  given by
 \[
  \hat \nabla u =\frac{1}{2}\left(\nabla u+(\nabla u)^T\right), \quad\quad
  \mathbb{C}\hat \nabla u=\left(\sum_{k,l=1}^d \mathbb{C}_{ijkl}\frac{\partial u_k}{\partial x_l}\right)_{1\leq i,j\leq d},
 \] 
 {\color{black}
 where 
 \begin{align*}
 \nabla u(x)=
\left(\frac{\partial}{\partial x_1}u(x),\cdots,\frac{\partial}{\partial x_d}u(x)\right),\qquad
\mathrm{div}F(x)=
\begin{pmatrix}
\sum_{i=1}^d \frac{\partial}{\partial x_i}F_{i1}(x)\\
\vdots\\
\sum_{i=1}^d \frac{\partial}{\partial x_i}F_{id}(x)
\end{pmatrix}
 \end{align*}
 \noindent
 for $F:\Omega\mapsto \mathbb{R}^{d\times d}$.
 }
The isotropic elastic   tensor  is defined as 
\begin{equation}
\label{tensor}
\mathbb{C}_{ijkl}:= \lambda\delta_{ij}\delta_{kl}+\mu\left(\delta_{ik}\delta_{jl}+\delta_{il}\delta_{jk}\right),
\end{equation}
where  $\lambda,\mu$ are the Lam\'e  coefficients.
\\
\\
{\color{black}
Next, we state a unique continuation principle from \cite{lin2011quantitative}:

\begin{theorem}[\color{black}Weak Unique Continuation Principle]\label{UCP}
Let $\Omega$ be a connected open set in $\mathbb{R}^d$ for $d\geq 2$.
Let $\mu(x)\in C^{0,1}(\Omega)$
and $\lambda(x)\in L^{\infty}(\Omega)$ satisfy
\begin{align*}
&\mu(x)\geq \delta_0, \quad \lambda(x)+2 \mu(x)\geq \delta_0\quad \mathrm{a.\,e.}\,\,x\in{\color{black} \Omega},\\
&\Vert \mu\Vert_{C^{0,1}({\color{black} \Omega})}+\Vert\lambda\Vert_{L^{\infty}(\Omega)}\leq M_0,
\end{align*}
\noindent
with positive constants $\delta_0$, $M_0$, where we define $\Vert f\Vert_{C^{0,1}(\Omega)}=\Vert f\Vert_{L^{\infty}({\color{black} \Omega})}+\Vert\nabla f\Vert_{L^{\infty}(\Omega)}$. 
If $u\in H^1(\Omega)^d$ satisfies
\begin{align*}
-\mathrm{div}(\mathbb{C}\hat \nabla u) &= 0\quad \text{ in }\,\Omega,\\
u&=0\quad \,\,\text{in some ball}\,\, B\subset\Omega,
\end{align*}
\noindent
then $u=0$ in $\Omega$.
\end{theorem}
}
\noindent
{\color{black}
\begin{proof}
The reader is referred to Theorem 1.2 from \cite{lin2011quantitative}.
\end{proof}
}

\begin{corollary}[Unique Continuation Principle for Local Cauchy Data]\label{UCPC}
{\color{black}
 Let the same assumptions as in Theorem \ref{UCP} hold and let $\Omega$ have a $C^{1,1}$ boundary and let $\Gamma$
 be a nonempty open subset of $\partial\Omega$. If $u\in H^1(\Omega)^d$ satisfies
 \begin{align*}
 -\mathrm{div}(\mathbb{C}\hat \nabla u) &= 0\quad \text{ in }\,\Omega,\\
  (\mathbb{C}\hat\nabla u)\nu&= 0\quad \text{ on  }\,\Gamma,\\
   u&= 0\quad \text{ on  }\,\Gamma,
 \end{align*}
 \noindent
 then $u=0$ in $\Omega$.
 }
\end{corollary}
 
{\color{black}
\begin{proof}
A solution with vanishing Cauchy data in $\Omega$ can be extended by zero to a solution in
an extended open set
$\tilde{\Omega}$, where $\Omega\subsetneq \tilde{\Omega}$. Hence, the weak UCP (Theorem \ref{UCP}) can be applied to show that $u=0$ in $\Omega$.
\end{proof}
}

\noindent
Next, for  given constants $\alpha_1, \alpha_2, \beta_1, {\color{black}\beta_2}$ satisfying  $0<\alpha_1\leq\alpha_2$,  $0<\beta_1\leq\beta_2$, 
we define the  set of admissible elastic {\color{black}tensors} by
\[
\mathcal{A}=\left\{\mathbb{C}= \mathbb{C}(\lambda,\mu): \quad  (\lambda,\mu)\in L^\infty(\Omega)\times C^{0,1}(\Omega), \quad 
\alpha_1\leq \lambda \leq \alpha_2,  \quad \beta_1\leq \mu \leq \beta_2
\right\}.
\]
\noindent
Hence, the Lam\'e parameters of every $\mathbb{C}(\lambda,\mu)\in\mathcal{A}$ satisfy the conditions of {\color{black}the unique continuation principles Theorem
\ref{UCP} and Corollary \ref{UCPC}}.
\\
\\
\noindent
In what follows,  we denote  $A:B=\sum_{i,j=1}^d a_{ij}b_{ij}$, for matrices $A= (a_{ij})$  and   $B= (b_{ij})$.
\\
\\
The weak  formulation  of  problem \eqref{direct} is given by
\begin{equation}
\label{var-direct}
\int_{\Omega}\mathbb{C}\hat\nabla u:\hat\nabla v \,dx=\int_{\Gamma_{\textup N}}g \cdot v \,ds \quad \text{ for all } v\in \mathcal{V},
\end{equation}
where
\[
\mathcal{V}:=\left\{   v\in H^1(\Omega)^d: \quad  v_{|_{\Gamma_{\textup D}}}=0\right\}.
\]
\noindent
It is easy to see that for each $\mathbb{C}\in \mathcal{A}$, problem  \eqref{var-direct} has a unique solution $u\in \mathcal{V}$, which follows by the Lax-Milgram theorem and is shown, e.g., {\color{black}in \cite{Ciarlet_book}.}
\\
\\
We introduce the  Neumann-to-Dirichlet operator $\Lambda(\mathbb{C})$:
\[
\Lambda(\mathbb{C}): L^2(\Gamma_{\textup N})^d\rightarrow L^2(\Gamma_{\textup N})^d: \quad  g\mapsto {u_\mathbb{C}^g{|_{\Gamma_{\textup N}}},}
\]
{where here and in the following $u_\mathbb{C}^g\in \mathcal{V}$ always denotes the unique solution of (\ref{direct}) with surface load 
$g\in L^2(\Gamma_{\textup N})^d$ and elastic tensor $\mathbb{C}\in \mathcal{A}$.}

It is well known {(and an easy consequence from the variational formulation \eqref{var-direct} and the compactness of the trace operator)}
that $ \Lambda(\mathbb{C})$ is a self-adjoint compact linear operator {that fulfills}
\[
\langle g,\Lambda(\mathbb{C})h\rangle=\int_\Omega\mathbb{C}\hat\nabla u^g_\mathbb{C}:\hat\nabla u^h_\mathbb{C}\,dx \quad  \text{ for all } g,h\in L^2(\Gamma_{\textup N})^d,
\]
{where $\langle \cdot, \cdot \rangle$ denotes the $L^2(\Gamma_{\textup N})$-inner product.}

The  inverse problem we  consider here   is the following:
\begin{equation}
\label{invp}
\text{ \it  Find }  \mathbb{C} \text{ or } (\lambda,\mu) \text{ \it knowing}\,\, \langle g,\Lambda(\mathbb{C})h\rangle.  
 \end{equation}

 \section{Monotonicity, localized potentials  and Lipschitz stability}\label{sec3}
 In this section, we  show a monotonicity estimate between the elastic tensor and the Neumann-to-Dirichlet operator    and the  existence of  localized potentials. Then we deduce a Lipschitz stability estimate {\color{black} for a finite dimensional subset with bounded Lam\'e parameters and a definiteness assumption.}
\begin{lemma}[Monotonicity estimate]
\label{mono}
Let $\mathbb{C}_1, \mathbb{C}_2 \in \mathcal{A}$ {be elastic tensors, and} $g\in L^2(\Gamma_{\textup N})^d$ be an applied boundary load. {The corresponding solutions of (\ref{direct}) are denoted by $u_1:=u^g_{\mathbb{C}_1},\ u_2:=u^{g}_{\mathbb{C}_2}\in \mathcal{V}$}. Then
\begin{equation}
\label{eqmono}
\int_\Omega(\mathbb{C}_1-\mathbb{C}_2)\hat\nabla u_2:\hat\nabla u_2\,dx
\geq \langle g,\Lambda(\mathbb{C}_2)g\rangle-\langle g,\Lambda(\mathbb{C}_1)g\rangle
\geq \int_\Omega(\mathbb{C}_1-\mathbb{C}_2)\hat\nabla u_1:\hat\nabla u_1\,dx.
\end{equation}
\end{lemma}
\begin{proof}
{Since $\Lambda(\mathbb{C}_2)g=u_2|_{\Gamma_{\textup N}}$ we can use the variational formulation \eqref{var-direct} for $\mathbb{C}_1$ and $\mathbb{C}_2$ with
$v:=u_2$ and obtain}
\[
\int_\Omega\mathbb{C}_1\hat\nabla u_1:\hat\nabla u_2\,dx=\langle g,\Lambda(\mathbb{C}_2)g \rangle
=\int_\Omega\mathbb{C}_2\hat\nabla u_2:\hat\nabla u_2\,dx.
\]
Thus 
\[
\begin{aligned}
\lefteqn{\int_\Omega\mathbb{C}_1\hat\nabla (u_1-u_2):\hat\nabla(u_1-u_2)\,dx}\\
& =\int_\Omega\mathbb{C}_1\hat\nabla u_1:\hat\nabla u_1\,dx
+\int_\Omega\mathbb{C}_1\hat\nabla u_2:\hat\nabla u_2\,dx
-2\int_\Omega\mathbb{C}_1\hat\nabla u_1:\hat\nabla u_2\,dx\\
&=\langle g,\Lambda(\mathbb{C}_1)g\rangle-\langle g,\Lambda(\mathbb{C}_2)g\rangle
+\int_\Omega(\mathbb{C}_1-\mathbb{C}_2)\hat\nabla u_2:\hat\nabla u_2\,dx.
\end{aligned}
\]
Since the left-hand side is nonnegative, the first asserted inequality follows. 
\\
Interchanging $\mathbb{C}_1$ and $\mathbb{C}_2$, 
{\color{black} the second inequality follows.}
\end{proof}
\noindent
{\color{black} Based on the previous lemma and the definition of $\mathbb{C}$, we are led to the following monotonicity property.}
\begin{corollary}[Monotonicity]\label{monotonicity}
For $\mathbb{C}_1:=\mathbb{C}(\lambda_1,\mu_1),  \mathbb{C}_2:=\mathbb{C}(\lambda_2,\mu_2) \in \mathcal{A}$
\begin{equation}\label{monotonicity_corol}
\lambda_1\leq \lambda_2 \text{ and } \mu_1\leq \mu_2 \quad \text{  implies } \quad \Lambda(\mathbb{C}_1)\geq \Lambda(\mathbb{C}_2).
\end{equation}
\end{corollary}
\begin{theorem}[Localized potentials]
\label{lemma:locpot}
Let $\mathbb{C}\in \mathcal{A}$  and {\color{black} $D_1, D_2$ be two  open sets with $\overline{D}_1, \overline{D}_2\Subset \Omega$}, 
 $\overline{D}_1\cap \overline{D}_2=\emptyset$ and let $\Omega\setminus(\overline{D}_1\cup \overline{D}_2)$ be connected. Then there exists a sequence
$(g_n)_{n\in \mathbb{N}}\subset L^2(\Gamma_{\textup N})^d$, such that the corresponding solutions
$(u^{ g_n})_{n\in \mathbb{N}}$ of \eqref{direct} fulfill 
\begin{align}
\label{localized_div_1}
&\lim_{n\to \infty}\int_{D_1} \left(\operatorname{div} u^{ g_n}\right)^2\,dx=\infty,\\ 
&\lim_{n\to \infty}\int_{D_2} \left( \operatorname{div}u^{ g_n}\right)^2\,dx=0, \label{localized_div_2}\\
&\lim_{n\to \infty}\int_{D_1} \hat \nabla u^{ g_n}:\hat \nabla u^{ g_n}\,dx=\infty,\label{localized_grad_1}\\
&\lim_{n\to \infty}\int_{D_2} \hat \nabla u^{ g_n}:\hat \nabla u^{ g_n}\,dx=0.\label{localized_grad_2}
\end{align}
\end{theorem}
\begin{proof}
{\color{black} This proof is based on the UCP for local Cauchy data (Corollary \ref{UCPC}).}
\noindent
First, we define the virtual  measurement operators 
\begin{itemize}
\item[(a)] $A_j$ ($j=1,2$) by
\begin{align*}
A_j :  {\color{black}L^2(D_j)}\rightarrow  L^2(\Gamma_{\textup N})^d, \quad  F\mapsto  v|_{\Gamma_{\textup N}},
\end{align*}
 where $v\in \mathcal{V}$ solves 
 \begin{align}\label{dual_A} 
 \int_\Omega\mathbb{C}\hat\nabla v:\hat\nabla w\,dx=\int_{D_j}  F \operatorname{div}w\,dx \quad \text{ for all } w\in \mathcal{V},
 \end{align}
 \item[(b)] $B_j$ ($j=1,2$) by
\begin{align*}
B_j : {\color{black} L^2(D_j)^{d \times d}}\rightarrow  L^2(\Gamma_{\textup N})^d, \quad G\mapsto  v|_{\Gamma_{\textup N}},
\end{align*}
 where $v\in \mathcal{V}$ solves 
 \begin{align}\label{dual_B}
 \int_\Omega\mathbb{C}\hat\nabla v:\hat\nabla w\,dx=\int_{D_j}  G :\hat{\nabla}w\,dx \quad \text{ for all } w\in \mathcal{V}.
 \end{align}
\end{itemize}
\noindent
First, we show that the dual operators
 \begin{align*}
&A'_j : L^2(\Gamma_{\textup N})^d  \rightarrow {\color{black}L^2(D_j)}, \quad  j=1,2,\\
&B'_j : L^2(\Gamma_{\textup N}) ^d \rightarrow  {\color{black} L^2(D_j)^{d \times d}}, \quad  j=1,2,
\end{align*}
are given  by  $A'_jg=\operatorname{div}(u)|_{D_j}$ and
$B'_j g=\hat{\nabla}u|_{D_j}$, where $u$  solves problem  \eqref{direct}.
\\
\begin{itemize}
\item[To (a):] Let {\color{black}$F\in L^2(\Omega)$}, $g\in L^2(\Gamma_{\textup N})^d$, {\color{black}$u$, $v\in \mathcal{V}$} solve (\ref{direct}) and (\ref{dual_A}), respectively. Then,
\begin{align*}
{\color{black} 
\int_{\Omega}F A^{\prime}_j g\,dx 
= \int_{\Gamma_{\textup{N}}}g\cdot A_j F\,ds
=\int_{\Omega}\mathbb{C}\hat{\nabla}v : \hat{\nabla} u\,dx
=\int_{D_j} F \mathrm{div}(u)\,dx.}
\end{align*}
\item[To (b):] Let {\color{black}$G\in L^2(\Omega)^{d\times d}$}, $g\in L^2(\Gamma_{\textup N})^d$, {\color{black}$u$, $v\in \mathcal{V}$} solve (\ref{direct}) and (\ref{dual_B}), respectively. Then,
\begin{align*}
{\color{black}
\int_{\Omega}G : B^{\prime}_j g \,dx
 = \int_{\Gamma_{\textup{N}}}g\cdot B_j G\,ds
 =\int_{\Omega}\mathbb{C}\hat{\nabla}v : \hat{\nabla} u\,dx
 =\int_{D_j}G: \hat{\nabla}u\,dx.}
\end{align*}
\end{itemize}
\noindent
{\color{black} Next, we will prove that} 
\begin{align}
{\color{black}
\mathcal{R}(A_1)\cap  \mathcal{R}(B_2)=\lbrace 0\rbrace \quad \mathrm{and}\quad \mathcal{R}(A_1)\neq \{0\} }.\label{range_A_1_B_2}
\end{align}
\\
Let  $\varphi \in \mathcal{R}(A_1)  \cap  \mathcal{R}(B_2)$.  Then there exist  $v_1, v_2
\in \mathcal{V}$  such  that $ v_1|_{\Gamma_{\textup N}}= v_2|_{\Gamma_{\textup N}} =\varphi$,
and 
\[
\int_\Omega\mathbb{C}\hat \nabla v_j:\hat \nabla w\,dx
=0
\]
for all $w\in \mathcal{V}$  with  $\text{supp}(w)\subset \overline{\Omega}\setminus\overline{D}_j$, $j=1,2$.  
\noindent
Hence,
 \[
\left\{
\begin{aligned}
\mathrm{div}(\mathbb{C}\hat \nabla v_1) & = 0&\quad &\text{ in }\Omega\setminus\overline{D}_1,\\
\mathrm{div}(\mathbb{C}\hat \nabla v_2) & = 0&\quad &\text{ in }\Omega\setminus\overline{D}_2,
\end{aligned}
\right.
\]
and $(\mathbb{C}\hat\nabla v_1)\nu|_{\Gamma_{\textup N}}=(\mathbb{C}\hat\nabla v_2)\nu|_{\Gamma_{\textup N}}=0$. 
{\color{black}The unique continuation principle for Cauchy data (Corollary \ref{UCPC})} yields that $v_1 = v_2$ in $\Omega\setminus(\overline{D}_1\cup\overline{D}_2)$. Hence  
{\color{black} $v:=  v_1\chi_{D_2}+v_2\chi_{\Omega\setminus \overline{D}_2}\in  \mathcal{V}$ } and  satisfies  
 \[
\left\{
\begin{aligned}
\mathrm{div}(\mathbb{C}\hat \nabla v) & = 0&\quad &\text{ in }\Omega,\\
(\mathbb{C}\hat \nabla v)\nu & = 0&\quad &\text{ on }\Gamma_{\textup N}.
\end{aligned}
\right.
\]
It follows  that $v=0$ and thus   $\varphi=v|_{\Gamma_{\textup N}}=0$, and consequently \mbox{$\mathcal{R}(A_1)  \cap  \mathcal{R}(B_2)=\left\{0\right\}$}.
\\
\\
{\color{black}
Next, we take a closer look at the operator $A_1:L^2(D_1)\to L^2(\Gamma_{\textup{N}})^d$, $F\mapsto v\vert_{\Gamma_{\textup{N}}}$.
Let $O$ be an open ball with $\overline{O}\subseteq D_1$ and let $v\in\mathcal{V}$ 
solve (\ref{dual_A}) for $F=\chi_O$. We will show that $v\vert_{\Gamma_{\textup{N}}}=A_1 \chi_O\neq 0$.
We argue by contradiction and assume $v\vert_{\Gamma_{\textup{N}}}=0$. 
From (\ref{dual_A}) we obtain that 
\begin{align*}
\int_{\Omega\setminus O} \mathbb{C} \hat{\nabla} v:\hat{\nabla} w \,dx= 0
\end{align*}
\noindent
for all $w\in \mathcal{V}$ with $\mathrm{supp}(w)\subseteq \overline{\Omega}\setminus\overline{O}$.
 Thus, $v$ fulfills
\begin{align*}
- \mathrm{div}( \mathbb{C}\hat{\nabla} v)=0 \,\,\,\textrm{in}\,\,\,\Omega\setminus \overline{O}
\,\,\,\mathrm{and}\,\,\, (\mathbb{C}\hat{\nabla}v)\nu=0 \,\,\, \mathrm{on}\,\,\, \Gamma_{\textup{N}}.
\end{align*}
\noindent
Now, we apply Corollary \ref{UCPC} on $\Omega\setminus \overline{O}$
which results in $v=0$ on $\Omega\setminus \overline{O}$, so that $v\vert_{O}\in H_0^1(O)^d$.
 Also, for all $w\in H_0^1(O)^d$, we obtain from (\ref{dual_A}) with $F=\chi_O$
\begin{align}\label{weak_zero}
\int_O \mathbb{C}\hat{\nabla} v :\hat{\nabla}w\, dx
=\int_{\Omega}\mathbb{C}\hat{\nabla}v:\hat{\nabla}\tilde{w}\,dx
=\int_O \mathrm{div}(w)\,dx
=\int_{\partial O}w\cdot \nu\,ds=0,
\end{align}
\noindent
where $\tilde{w}\in H_0^1(\Omega)^d$ is the zero extension of $w\in H_0^1(O)^d$. Since (\ref{weak_zero}) is uniquely solvable 
by the Lax-Milgram-Theorem it follows that $v\vert_O=0$, so that $v=0$ in all of $\Omega$ which contradicts (\ref{dual_A}) with $F=\chi_O$. 
Hence, $A_1\chi_O\neq 0$ which shows that $\mathcal{R}(A_1)\neq \{0\}$.
\\
This proves (\ref{range_A_1_B_2}) and this implies $\mathcal{R}(A_1)\not\subseteq \mathcal{R}(B_2)$.
Using [\cite{gebauer2008localized}, Corollary 2.6] it follows that there exists a sequence
$(g_n)_{n\in \mathbb{N}}\subset L^2(\Gamma_{\textup N})^d$:
\begin{align*}
\lim_{n\to \infty}\Vert A_1^\prime g_n\Vert^2_{L^2(\Omega)}=\lim_{n\to \infty}\int_{D_1} \left(\operatorname{div}u^{ g_n}\right)^2\,dx=\infty
\end{align*}
\noindent
and
\begin{align*}
\lim_{n\to \infty}\Vert B_2^\prime g_n\Vert^2_{L^2(\Omega)^{d\times d}}=\lim_{n\to \infty}\int_{D_2}  \hat{\nabla}u^{ g_n} : \hat{\nabla}u^{ g_n} \,dx=0,
\end{align*}
\noindent
i.e. (\ref{localized_div_1}) and (\ref{localized_grad_2}) hold.
Since
\begin{align*}
\mathrm{tr}\left( \hat{\nabla} u^{ g_n}\right)=\mathrm{div}u^{ g_n}
\end{align*}
\noindent
this also implicates (\ref{localized_div_2}) and (\ref{localized_grad_1}).
}
\end{proof}
\noindent
Next, we go over to the background of the Lipschitz stability and introduce the definition of piecewise continuous functions.
\begin{definition}
A function $f\in L^\infty(\Omega)$ is called piecewise continuous, if there exists a finite decomposition of $\Omega$ {into non-empty open subsets $\Omega_i\subseteq \Omega$, $i=1,...,n$, so that} $ \Omega \setminus \bigcup_{i=1}^n\Omega_i $ is a {Lebesgue} null set, $\Omega_i \cap \Omega_j=\emptyset$ ($i\neq j$), { and $f|_{\Omega_i}$ is continuous for all $i=1,...,n$.}
\end{definition}

{Inverse elliptic coefficient problems are known to be ill-posed and stability results can only be obtained under a-priori assumptions, cf.\ the works cited in the introduction. For our problem, we will prove a stability result under the assumption that the coefficients belong to an a-priori known finite-dimensional subspace (e.g., stemming from 
the parameter parametrization or a desired finite resolution), that upper and lower bounds are a-priori known, and that a definiteness condition holds.
More precisely, let} $\mathcal{F}$ be a finite dimensional {subspace} of $\check{C}(\Omega)\times C^{0,1}(\Omega)$, where $\check{C}(\Omega)$ is the space
of piecewise continuous functions. 
We consider four constants \mbox{$0<a\leq b$} and \mbox{$0< c \leq d$} {\color{black} which are the lower and upper bounds of the Lam\'e parameter} and define the sets
\[
\mathcal{F}_{[a,b]\times[c,d]}=\left\{ (\lambda,\mu)\in \mathcal{F}:  \quad  a\leq\lambda(x)\leq b,  \quad  c\leq\mu(x)\leq d \quad \text{ for all } x\in \Omega \right\},
\]
\[
\mathcal{E}= \left\{ \mathbb{C}(\lambda, \mu): \quad  (\lambda, \mu)\in  \mathcal{F}_{[a,b]\times[c,d]} \right\}.
\]
{\color{black} In the following main result of this paper, the domain $\Omega$, the finite-dimensional {subspace} $\mathcal{F}
$ and the bounds $0< a\leq b$ and $0< c \leq d$ are fixed, and the constant in the Lipschitz stability result
 will depend on them.}

\begin{theorem}[Lipschitz stability]
\label{stability}
{\color{black}  There exists a positive constant $C>0$ such that for all 
$\mathbb{C}_1:=\mathbb{C}(\lambda_1,\mu_1),\mathbb{C}_2:=\mathbb{C}(\lambda_2,\mu_2) \in\mathcal{E}$ with either
\begin{align*}
&(a)\quad  \lambda_1\leq \lambda_2 \,\,\textrm{and}\,\, \mu_1\leq \mu_2\quad \textrm{or}\\
&(b)\quad \lambda_1\geq \lambda_2 \,\,\textrm{and}\,\, \mu_1\geq \mu_2,
\end{align*}
\noindent
we have 
\begin{equation}
\label{stab-est}
d_\Omega(\mathbb{C}_1,\mathbb{C}_2):=\mathrm{max}\left\{ \Vert\lambda_1-\lambda_2\Vert_{L^\infty(\Omega)},  \Vert\mu_1-\mu_2\Vert_{L^\infty(\Omega)}\right\} 
\leq   C\| \Lambda(\mathbb{C}_1)-\Lambda(\mathbb{C}_2) \|_*.
\end{equation}
Here  $\Vert.\Vert_*$ is the natural norm of $\Vert.\Vert_{\mathcal L(L^2(\Gamma_{\textup N}))}$. 
}
\end{theorem}

\begin{proof}

It suffices to prove the theorem for assumption (b) since the other case follows from interchanging $(\lambda_1,\mu_1)$ and $(\lambda_2,\mu_2)$.
For the sake of brevity, we write  $\Vert\cdot \Vert$  for  $\Vert \cdot \Vert_{L^2(\Gamma_{\textup N})^d}$.
We start with the reformulation of the right-hand side of estimate (\ref{stab-est}).
Since $\Lambda(\mathbb{C}_1)$ and $\Lambda(\mathbb{C}_2)$ are self-adjoint,
 and assumption (b) implies $\Lambda(\mathbb{C}_2)\geq \Lambda(\mathbb{C}_1)$ by 
Corollary~\ref{monotonicity}, we have that
\begin{align*}
\lefteqn{\Vert\Lambda(\mathbb{C}_2)-\Lambda(\mathbb{C}_1)\Vert_*}\\
&= 
  \sup_{\Vert g\Vert=1} \left| \langle g, \left(\Lambda(\mathbb{C}_2)-\Lambda(\mathbb{C}_1)\right) g\rangle\right|
= \sup_{\Vert g\Vert=1} \langle g, \left(\Lambda(\mathbb{C}_2)-\Lambda(\mathbb{C}_1)\right) g\rangle.
\end{align*}

Next, we apply the second inequality in the monotonicity relation \eqref{eqmono} in Lemma \ref{mono} and thus obtain 
for all $g\in L^2(\Gamma_{\textup N})^d$
 \begin{align}\label{estim_1}
 \nonumber 
\lefteqn{\langle g,\left(\Lambda(\mathbb{C}_2)-\Lambda(\mathbb{C}_1)\right) g\rangle}\\
&\geq \int_{\Omega} (\mathbb{C}_1-\mathbb{C}_2)\hat\nabla u_{\mathbb{C}_1}^{g}:\hat\nabla u_{\mathbb{C}_1}^{g}\,dx,\\
&=\int_\Omega(\lambda_1-\lambda_2)\left(\mathrm{div} u^{g}_{(\lambda_1,\mu_1)}\right)^2\,dx
\nonumber \quad {} +2\int_\Omega(\mu_1-\mu_2)\hat\nabla u_{(\lambda_1,\mu_1)}^{g}:\hat\nabla u_{(\lambda_1,\mu_1)}^{g}\,dx
 \end{align}
where $u_{\mathbb{C}_1}^{g}=u_{(\lambda_1,\mu_1)}^{g}\in \mathcal{V}$ denotes the solution of \eqref{direct} with Neumann data $g$ and elastic tensor  $\mathbb{C}_1:=\mathbb{C}(\lambda_1,\mu_1)$.

The estimate (\ref{estim_1}) contains the linear differences $\lambda_2-\lambda_1$ and $\mu_2-\mu_1$,
but it also contains the solution $u_{(\lambda_1,\mu_1)}^{g}$ that depends non-linearly on the coefficients.
Following the ideas of \cite{harrach2018global, harrach2019uniqueness}, we separate these dependencies by introducing 
\begin{align*}
\Psi&:\ L^2(\Gamma_{\textup N})^d\times \mathcal{F}\times \mathcal{F}_{[a,b]\times[c,d]}\to \mathbb{R}\\
\Psi\left(g,(\zeta_1,\zeta_2),(\kappa,\tau)\right)&:=\int_\Omega \zeta_1\left(\mathrm{div} u^{g}_{(\kappa,\tau)}\right)^2\,dx+
2\int_{\Omega} \zeta_2 \hat\nabla u_{(\kappa,\tau)}^{g}: \hat\nabla u_{(\kappa,\tau)}^{g}\,dx.
\end{align*}

Thus, we obtain for $\mathbb{C}_1\neq \mathbb{C}_2$,
\begin{align}\label{estim_3}
\frac{\Vert \Lambda(\mathbb{C}_2)-\Lambda(\mathbb{C}_1) \Vert_*}{d_\Omega(\mathbb{C}_1,\mathbb{C}_2)}\geq  
 \sup_{\Vert g\Vert=1}\Psi\left(g,\left(\frac{\lambda_1-\lambda_2}{d_\Omega(\mathbb{C}_1,\mathbb{C}_2)},\frac{\mu_1-\mu_2}{d_\Omega(\mathbb{C}_1,\mathbb{C}_2)}\right),(\lambda_1,\mu_1)\right).
\end{align}

By assumption (b), and the definition of $d_\Omega(\mathbb{C}_1,\mathbb{C}_2)$, it follows that the second argument of $\Psi$ in (\ref{estim_3}) belongs to the compact set
\begin{align*}
 {\mathcal{C}}&:=\left\{ (\zeta_1,\zeta_2)\in \mathcal{F}:  \quad  \zeta_1,\zeta_2\geq 0 \quad \text{  and }\quad 
 \max\left(\Vert \zeta_1\Vert_{L^\infty(\Omega)}, \Vert \zeta_2\Vert_{L^\infty(\Omega)}\right)= 1 \right\}.
\end{align*}
Hence, (\ref{estim_3}) yields that
\begin{equation}\label{eq:stability_infsup}
\frac{\| \Lambda(\mathbb{C}_2)-\Lambda(\mathbb{C}_1) \Vert_*}{d_\Omega(\mathbb{C}_1,\mathbb{C}_2)}\geq
\inf_{\substack{(\zeta_1, \zeta_2)\in \mathcal{C},\\ (\kappa,\tau)\in\mathcal{F}_{[a,b]\times[c,d]}}} \sup_{\| g \|=1} \Psi\left(g,(\zeta_1,\zeta_2),(\kappa,\tau)\right).
\end{equation}

The assertion of Theorem \ref{stability} follows if we can show that the right-hand side of \eqref{eq:stability_infsup}
is strictly positive. Since $\Psi$ is continuous, we can conclude that the function
\[
\left( (\zeta_1,\zeta_2),(\kappa,\tau) \right)\mapsto \sup_{\| g \|=1} \Psi\left(g,(\zeta_1,\zeta_2),(\kappa,\tau)\right)
\]
is semi-lower continuous, so that it attains its minimum on  the compact set \\
$\mathcal{C}\times \mathcal{F}_{[a,b]\times[c,d]}$.
Hence, to prove Theorem \ref{stability}, it suffices to show that
\begin{align}\label{estim_4}
\sup_{\| g \|=1} \Psi\left(g,(\zeta_1,\zeta_2),(\kappa,\tau)\right)>0
\quad \text{ for all }  (\zeta_1,\zeta_2)\in \mathcal{C}, \ (\kappa,\tau)\in \mathcal{F}_{[a,b]\times[c,d]}.
\end{align}

In order to prove that (\ref{estim_4}) holds true, let $\left((\zeta_1,\zeta_2),(\kappa,\tau)\right)\in \mathcal{C}\times \mathcal{F}_{[a,b]\times[c,d]}$.
Then there exist an open subset $\emptyset\neq D_1\subset \Omega$ and  a constant $0<\delta<1$,  such that either
\[
\begin{aligned}
&\text{(i)}\   \zeta_1|_{D_1}\geq \delta,  \text{ and } \zeta_2\geq 0, \text{ or }\\
&\text{(ii)}\   \zeta_2|_{D_1}\geq \delta,  \text{ and } \zeta_1\geq 0.
\end{aligned}
\]
We use the localized potentials sequence in Theorem~\ref{lemma:locpot} to obtain an open subset $D_2\subset \Omega$  with  $\overline{D_1}\cap \overline{D_2}=\emptyset$, and a  boundary load
$\tilde g\in L^2(\Gamma_{\textup N})^d$ with 
\begin{align}\label{estim_loc_pot}
\int_{D_1}  \left(\operatorname{div} u^{\tilde g}_{(\kappa,\tau)}\right)^2\,dx  \geq \frac{1}{\delta} \quad
 \text{ and } \quad \int_{D_1}\hat \nabla  u^{\tilde g}_{(\kappa,\tau)}: \hat \nabla  u^{\tilde g}_{(\kappa,\tau)}\,dx  \geq \frac{1}{2\delta }.
\end{align}

In case (i), this leads to
\[
\begin{aligned}
&\Psi\left(\tilde g,(\zeta_1,\zeta_2),(\kappa,\tau)\right)\\
&\geq \int_\Omega \zeta_1\left(\mathrm{div} u^{\tilde g}_{(\kappa,\tau)}\right)^2\,dx+
2 \int_\Omega \zeta_2 \hat\nabla  u^{\tilde g}_{(\kappa,\tau)}:\hat\nabla  u^{\tilde g}_{(\kappa,\tau)}\, dx \\
& \geq \int_{D_1} \zeta_1\left(\mathrm{div} u^{\tilde g}_{(\kappa,\tau)}\right)^2\,dx
 \geq \delta\int_{D_1} \left(\mathrm{div} u^{\tilde g}_{(\kappa,\tau)}\right)^2\,dx \geq 1
 \end{aligned}
\]
and in case (ii), we  have
\[
\begin{aligned}
&\Psi\left(\tilde g,(\zeta_1,\zeta_2),(\kappa,\tau)\right)\\
&\geq \int_\Omega \zeta_1\left(\mathrm{div} u^{\tilde g}_{(\kappa,\tau)}\right)^2\,dx+
 2 \int_\Omega \zeta_2 \hat\nabla  u^{\tilde g}_{(\kappa,\tau)}:\hat\nabla  u^{\tilde g}_{(\kappa,\tau)}\, dx \\
&\geq  2 \int_{D_1} \zeta_2 \hat\nabla  u^{\tilde g}_{(\kappa,\tau)}:\hat\nabla  u^{\tilde g}_{(\kappa,\tau)}\, dx
 \geq 2\delta \int_{D_1} \hat\nabla  u^{\tilde g}_{(\kappa,\tau)}:\hat\nabla  u^{\tilde g}_{(\kappa,\tau)}\, dx  \geq 1.
 \end{aligned}
\]

Hence, in both  cases, 
\[
\begin{aligned}
\sup_{\| g \|=1} \Psi(g,(\zeta_1,\zeta_2),(\kappa,\tau))
&\geq  \Psi\left(\frac{\tilde g}{\Vert \tilde g\Vert},(\zeta_1,\zeta_2),(\kappa,\tau)\right)\\
&=\frac{1}{\Vert\tilde g\Vert^2} \Psi(\tilde g,(\zeta_1,\zeta_2),(\kappa,\tau))>0,  
\end{aligned}
\]
so that Theorem \ref{stability} is proven.
\end{proof}

{\color{black}
\begin{remark} 
As a consequence of Theorem \ref{stability}, we end up with the following uniqueness result for 
all $\mathbb{C}_1:=\mathbb{C}(\lambda_1,\mu_1),\mathbb{C}_2:=\mathbb{C}(\lambda_2,\mu_2)$ {\color{black} 
from the finite dimensional space $\mathcal{E}$},
provided that either, $\lambda_1-\lambda_2\leq 0$ and $\mu_1-\mu_2\leq 0$, or $\lambda_1-\lambda_2\geq  0$ and $\mu_1-\mu_2\geq0$:
\[
\Lambda(\mathbb{C}_1)=\Lambda(\mathbb{C}_2)\quad \text{ if and only if }\quad \mathbb{C}_1=\mathbb{C}_2.
\]
\end{remark}
}
{\color{black}
\begin{remark}\label{DtN}
 All of the results in this section stay valid (with the same proofs) when the Neumann-to-Dirichlet operator $\Lambda(\mathbb{C})$
 is extended to the spaces $H^{-\frac{1}{2}}(\Gamma_{\textup{N}})\to H^{\frac{1}{2}}(\Gamma_{\textup{N}})$, where
 $H^{\frac{1}{2}}(\Gamma_{\textup{N}})=\lbrace u\vert_{\Gamma_{\textup{N}}}:u\in\mathcal{V}\rbrace$ and $H^{-\frac{1}{2}}(\Gamma_{\textup{N}})$
 is its dual. On these spaces, it is easily shown that $\Lambda(\mathbb{C})$ is bijective, and its inverse is the
 Dirichlet-to-Neumann operator $\Lambda_{\textup{D}}: f\mapsto u_{\textup{D}}\vert_{\Gamma_{\textup{N}}}$, where 
 $u_{\textup{D}}$ solves
 \begin{equation}
 \left\{
 \begin{aligned}
   -\text{div}(\mathbb{C}(\lambda,\mu)\hat\nabla u_{\textup D}) & = 0 &\quad& \text{ in }\Omega,\\
  u_{\textup D}&= f &\quad &\text{ on  }\Gamma_{\textup N},\\
  u_{\textup D}&= 0&\quad &\text{ on  }\Gamma_{\textup D}.\\
\end{aligned}
\right.
\end{equation}
\end{remark}
}

 \section{Numerical approach to solve the inverse problem}\label{sec4}
{\color{black} In this section, we consider Lam\'e parameters $(\lambda,\mu)\in \tilde{\mathcal{F}}$, 
where $\tilde{\mathcal{F}}$  is a finite dimensional subset of $L^\infty(\Omega)\times L^\infty(\Omega)$-functions
with positive minima.
}
\\
{\color{black} 
We first take a look at the inverse problem
\begin{equation}
\label{invp_num}
\text{ \it  Find }  \mathbb{C}(\lambda,\mu) \text{ \it knowing measurements}\,\,f_k=\Lambda(\mathbb{C}(\lambda,\mu))g_k,
\,\,k=1,\ldots K,
 \end{equation}
 \noindent
where $f_k\in L^2(\Gamma_{\textup N})^d$ is a measurement of the displacement corresponding to the
input surface load $g_k$, and $K\in \mathbb{N}$ is the number of measurements. 
\\
\\
To solve the inverse problem (\ref{invp_num}) numerically,
we first consider a single measurement $(f,g)$ and define a minimization problem of Kohn-Vogelius type:
\begin{align}\label{func-J}
\min_{(\lambda,\mu)\in \tilde{\mathcal{F}}}J(\lambda,\mu)=  
&\int_{\Omega} \mathbb{C}(\lambda,\mu)\hat \nabla (u_{\textup N}-u_{\textup D}):\hat \nabla(u_{\textup N}
-u_{\textup D}) \, dx.
\end{align}
\noindent
Here $u_{\textup N}$  and $u_{\textup D} $  solve  the following problems:
\begin{equation}\label{neum}
\left\{
\begin{aligned}
  -\text{div}(\mathbb{C}(\lambda,\mu)\hat\nabla u_{\textup N} ) & = 0&\quad &\text{ in }\Omega,\\
  (\mathbb{C}(\lambda,\mu)\hat \nabla u_{\textup N})\nu&= g & \quad&\text{ on  }\Gamma_{\textup N},\\
   u_{\textup N}&= 0& \quad&\text{ on  }\Gamma_{\textup D},
\end{aligned}
\right.
\end{equation}
\begin{equation}
\label{dir}
\left\{
\begin{aligned}
   -\text{div}(\mathbb{C}(\lambda,\mu)\hat\nabla u_{\textup D}) & = 0 &\quad& \text{ in }\Omega,\\
  u_{\textup D}&= f &\quad &\text{ on  }\Gamma_{\textup N},\\
  u_{\textup D}&= 0&\quad &\text{ on  }\Gamma_{\textup D}.\\
\end{aligned}
\right.
\end{equation}
 }
\noindent

\begin{theorem}\label{diffrobin}
{\color{black}
 The functional $J:L^\infty_+(\Omega) \times L^\infty_+(\Omega)\to \mathbb{R}$, defined in (\ref{func-J})
 is Fr\'echet differentiable, and  its Fr\'echet derivative  at  
 $(\lambda,\mu)\in  L^\infty_+(\Omega)\times L^\infty_+(\Omega)$ in the  
 direction  $(\tilde \lambda,\tilde\mu)\in L^\infty_+(\Omega)\times L^\infty_+(\Omega)$ is given by
 \begin{equation}
 \label{DJ}
 \begin{aligned}
  J^\prime\left(\lambda,\mu\right)(\tilde{\lambda},\tilde{\mu})
 =&\int_\Omega\mathbb{C}(\tilde{\lambda},\tilde{\mu})\hat\nabla u_{\textup D}:\hat\nabla u_{\textup D}\,dx
  -\int_\Omega\mathbb{C}(\tilde{\lambda},\tilde{\mu})\hat\nabla u_{\textup N}:\hat\nabla u_{\textup N}\,dx.
 \end{aligned}
 \end{equation}
 }
\end{theorem}
\begin{proof}
{\color{black}
From the definition of the functional $J$, and applying Green's formula once, we have 
\begin{align*}
 J(\lambda,\mu)
 &=\int_{\Omega}\mathbb{C}(\lambda,\mu)\hat{\nabla}u_{\textup{N}}:\hat{\nabla}u_{\textup{N}}\,dx+
 \int_{\Omega}\mathbb{C}(\lambda,\mu)\hat{\nabla}u_{\textup{D}}:\hat{\nabla}u_{\textup{D}}\,dx
 -2 \int_{\Gamma_{\textup{N}}}g\cdot f\,ds\\
 &=\langle g,\Lambda(\mathbb{C}(\lambda,\mu))g\rangle + \langle \Lambda_{\textup{D}}(\mathbb{C}(\lambda,\mu))f,f\rangle 
 -2 \int_{\Gamma_{\textup{N}}}g\cdot f\,ds.
\end{align*}
\noindent
Since $\Lambda(\mathbb{C}(\lambda,\mu))$ is Fr\'echet differentiable with
\begin{align*}
 \langle g,\Lambda^\prime(\mathbb{C}(\lambda,\mu))(\tilde{\lambda},\tilde{\mu})g\rangle
 =-\int_{\Omega}\mathbb{C}(\tilde{\lambda},\tilde{\mu})\hat{\nabla}u_{\textup{N}}:\hat{\nabla}u_{\textup{N}}\,dx
\end{align*}
\noindent(c.f., e.g. \cite{Lechleiter} for a proof that uses only the variational formulation) and
\begin{align*}
 \Lambda_{\textup{D}}(\mathbb{C}(\lambda,\mu))=\Lambda(\mathbb{C}(\lambda,\mu))^{-1}
\end{align*}
\noindent
implies
\begin{align*}
 \langle \Lambda_{\textup{D}}^\prime(\mathbb{C}(\lambda,\mu))(\tilde{\lambda},\tilde{\mu})f,f\rangle&=
 -\langle \Lambda(\mathbb{C}(\lambda,\mu))^{-1}f,\Lambda^\prime(\mathbb{C}(\lambda,\mu))(\tilde{\lambda},\tilde{\mu})
 \Lambda(\mathbb{C}(\lambda,\mu))^{-1}f\rangle\\
 &=\int_{\Omega}\mathbb{C}(\tilde{\lambda},\tilde{\mu})\hat{\nabla}u_{\textup{D}}:\hat{\nabla}u_{\textup{D}}\,dx
\end{align*}
\noindent
(see Remark \ref{DtN}), the assertion follows.
}
\end{proof}
\noindent
{\color{black}
In the next section, we treat the case of several measurements $(f_1,g_1),\ldots,(f_K,g_K)$ by adding the respective
functionals of the form (\ref{func-J}), together with a regularization term.
}

\section{Implementation details and numerical examples}\label{sec5}
{\color{black} As stated in the introduction, our aim is to obtain first results
for the solution of the inverse problem of elasticity. Thus, we examine the intuitive 2 dimensional setting as a first descriptive and meaningful example. 
In doing so, we consider the following setup for our numerical example:}
The domain $\Omega$  under consideration is the {\color{black} two dimensional}  unit disk centered at  the origin.
We  use a Delaunay triangular mesh and a   standard finite element method with piecewise finite elements to numerically compute  the states  for our problem. 
 The exact  data $f$  are computed  synthetically by solving the direct problem \eqref{direct}. 
 In   the real-world, the data $f$  are experimentally acquired and thus {\color{black} can be corrupted by noise such as arising from quantization errors.}
 
In our numerical examples the simulated noise data  are generated using the following formula:
\[
\tilde{f}(x_1,x_2)= f(x_1,x_2)\left(1+\varepsilon\delta\right)\quad \text{ on } \Gamma_{\textup N},
\]
where $\delta$ is a uniform distributed random variable and $\varepsilon$ indicates the level of noise.  For our examples, the random variable $\delta$ is realized using 
the Matlab   function rand$()$. 
 We use the BFGS algorithm {\color{black}\cite{hong}},  {\color{black}from  the optimization toolbox of Matlab},   to minimize  the cost function defined in  \eqref{func-J}.  This quasi-Newton method is well adapted  to such problem.   
\subsection{Numerical examples}
For the following numerical examples,  we  use four  measurements 
{\color{black}
$f_1,\ldots,f_4$}
corresponding to the following surface  loads:
\[
g_1=(0.1,0.1), \quad g_2=(0.1,0.2),\quad g_3=(0.2,0.1), \quad \text{ and } \quad g_4=(0.3,0.5)
 \text{ on }\Gamma_{\textup N}. 
\] 
{\color{black}
We reconstruct the Lam\'e parameters by minimizing the regularized cost functional
\begin{align}\label{J_reg}
J(\lambda,\mu):=  \sum_{k=1}^4\int_{\Omega} \mathbb{C}(\lambda,\mu)\hat \nabla (u^{g_k}_{\textup N}-u^{f_k}_{\textup D}):
\hat \nabla(u^{g_k}_{\textup N}-u^{f_k}_{\textup D}) \, dx+{\color{black}\frac{\rho}{2}\int_\Omega \left(\lambda^2+\mu^2\right)\,dx},
\end{align}
where  $u^{g_k}_{\textup N}$ and $u^{f_k}_{\textup D}$  are  the solutions to problem \eqref{neum} and \eqref{dir}
respectively  with respect to the boundary load $g_k$ and   the corresponding measurement data $f_k$. 
Here, $\rho$ is a heuristically chosen regularization parameter. The derivative of (\ref{J_reg}) is easily obtained
from Theorem \ref{diffrobin}.
}

\subsubsection{Example 1}
In this example, the  Lam\'e parameters  to  be recovered  are  assumed  to be {constant on $\Omega$} and we consider the minimization problem of recovering two scalar parameters.
Let $(\lambda_i,\mu_i){\in \mathbb{R}^2}$  denote the initialization,
 $(\lambda_e,\mu_e){\in \mathbb{R}^2}$ the exact parameters to be recovered and  $(\lambda_c,\mu_c){\in \mathbb{R}^2}$ the computed parameters.
 {\color{black} }
 Table \ref{lame} summarizes the computational results of the algorithm. Figures \ref{geom1}-\ref{geom3}, show  
 the decrease of the cost function $J$  and the $L^\infty$-norm of $J^\prime$ in the course of the optimization process.
 The numerical solution represents a good  approximation and it is stable with respect to small amounts of noise.

 \begin{table}[h]
\begin{center}
 \begin{tabular}{ |c|c|c|c|c|c| c|}
\hline
 &$(\lambda_i,\mu_i)$ & $(\lambda_c,\mu_c)$ & $(\lambda_e,\mu_e)$& $ \frac{|\lambda_c-\lambda_e|}{|\lambda_e |}$   & $\frac{|\mu_c-\mu_e|}{|\mu_e |}$ \\
\hline
 $\epsilon=0.0,\; \rho=0.0$&(1,1)& (2.9999, 6.9999) & (3,7) & 3.33e-07&1.42e-07\\
\hline
 $\epsilon=0.03,\; \rho=0.00001$&(1,1)& (2.6564, 6.8989) & (3,7) & 0.1145& 0.0144\\
 \hline
 $\epsilon=0.05,\; \rho=0.00001$&(1,1)& ( 2.6527,6.8869) & (3,7) & 0.1157&0.0161\\
 \hline
\end{tabular}
\end{center} 
\caption{{\color{black}Simulation results for Example 1: Reconstruction of constant} Lam\'e parameters.}\label{lame}
\end{table}
 
\noindent
\begin{figure}[ht!]
\begin{center}
 \begin{tabular}{c@{\qquad} c}
 \includegraphics[width=0.45\textwidth]{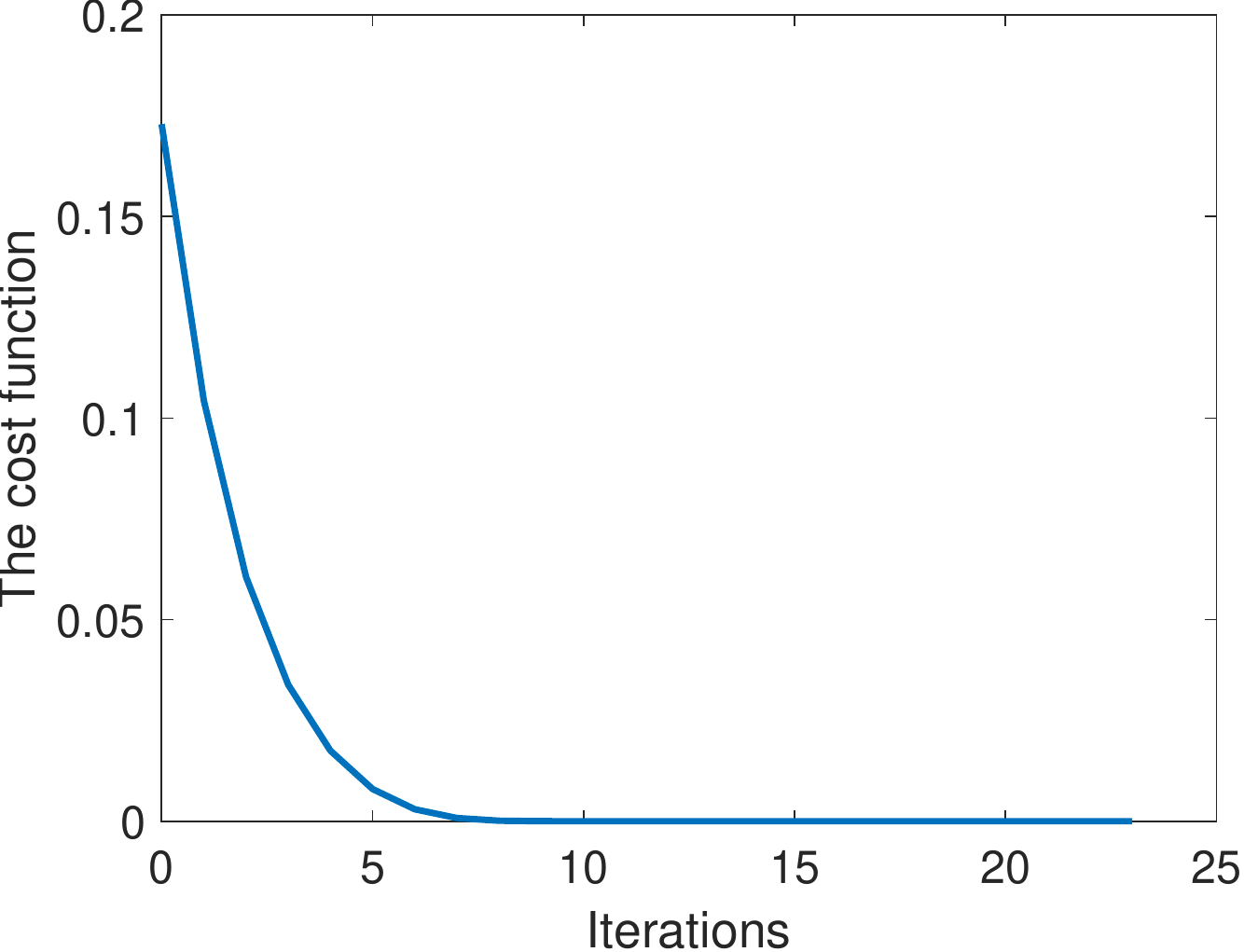} &
\includegraphics[width=0.45\textwidth]{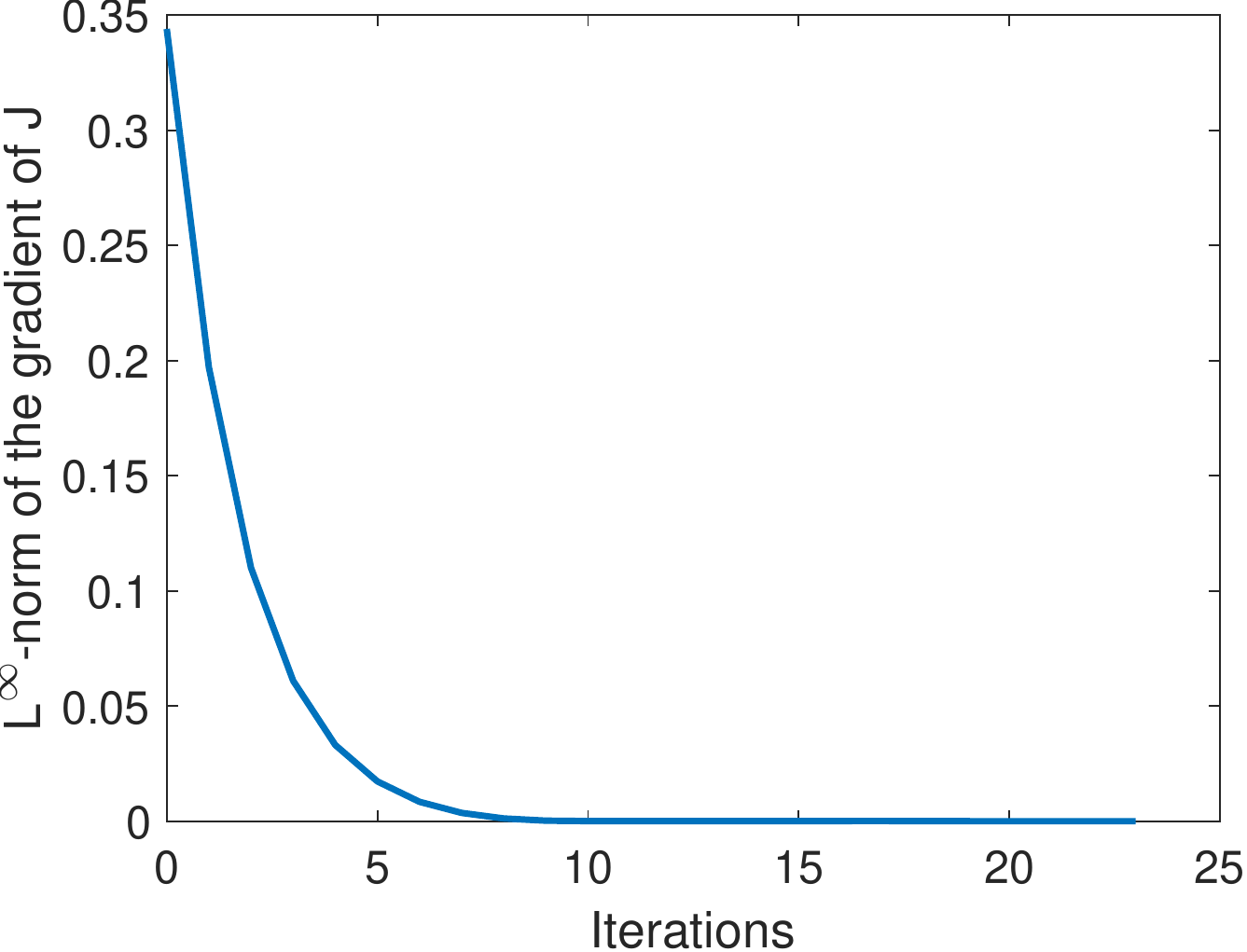}
\end{tabular}
 \caption{Simulation results for Example 1:  History of the cost function $J$ and the $L^\infty$-norm of $J^\prime$  in the case of  
   $\varepsilon=0.0$ and $\rho=0.0$.}
  \label{geom1} 
\end{center}
 \end{figure}


 \begin{figure}[ht!] 
 \begin{center}
 \begin{tabular}{c@{\qquad} c}
\includegraphics[width=0.45\textwidth]{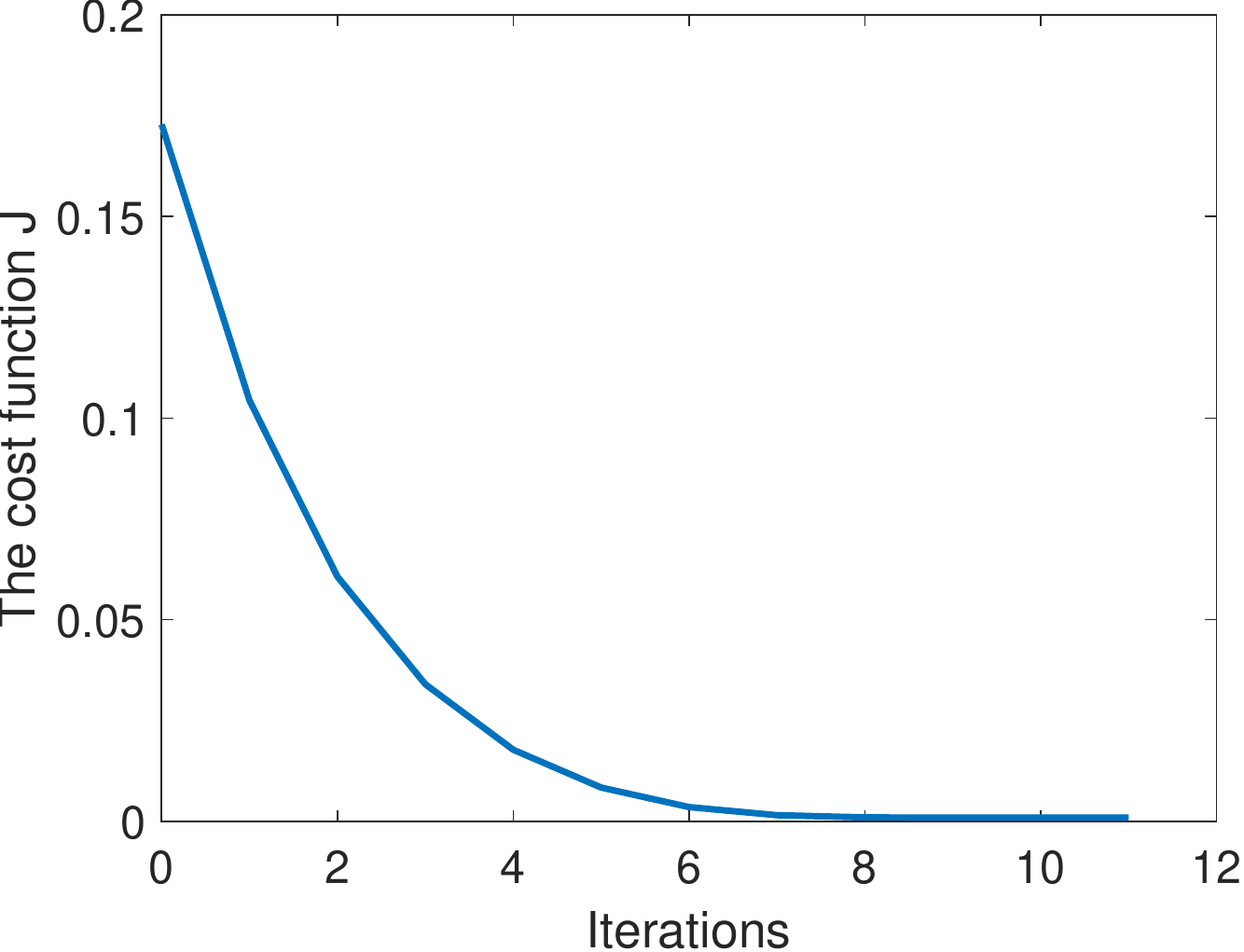} &
\includegraphics[width=0.45\textwidth]{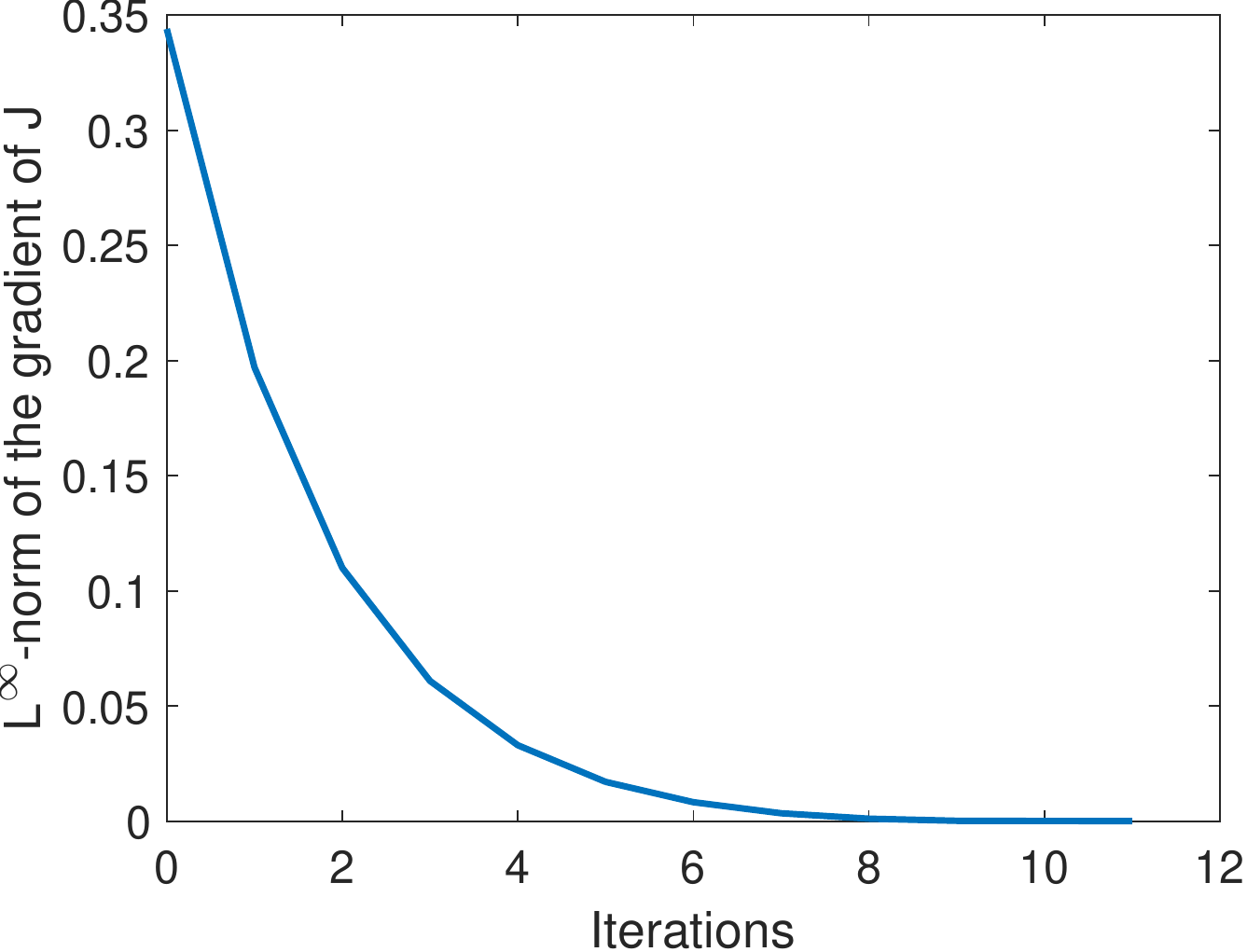}
\end{tabular}
 \caption{Simulation results for Example 1: History of the cost function $J$ and the $L^\infty$-norm of $J^\prime$  in the case of  
   $\varepsilon=0.03$ and $\rho=0.00001$.}
  \label{geom2} 
  \end{center}
 \end{figure}


 \begin{figure}[ht!] 
 \begin{center}
  \begin{tabular}{c@{\qquad} c}
\includegraphics[width=0.45\textwidth]{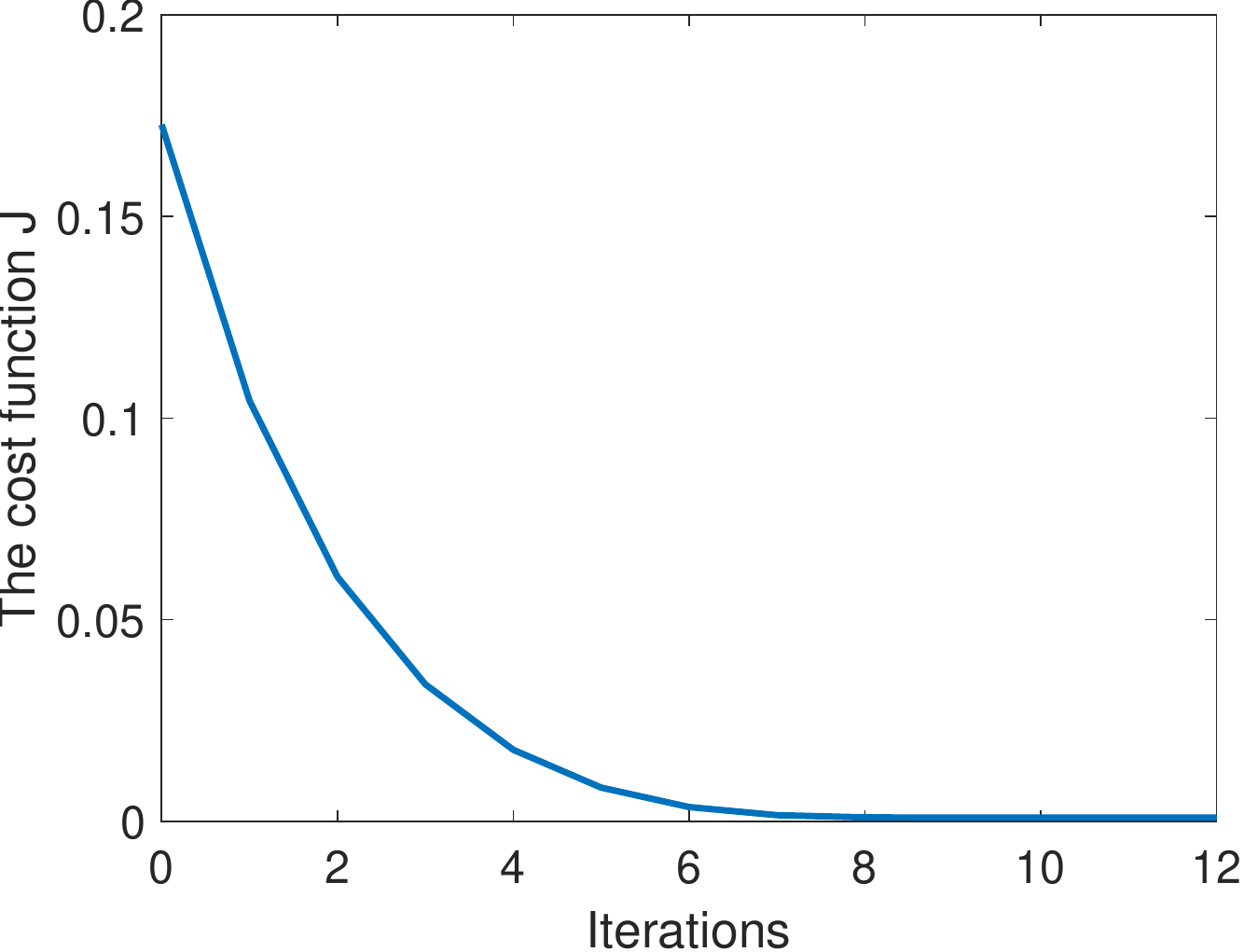} &
\includegraphics[width=0.45\textwidth]{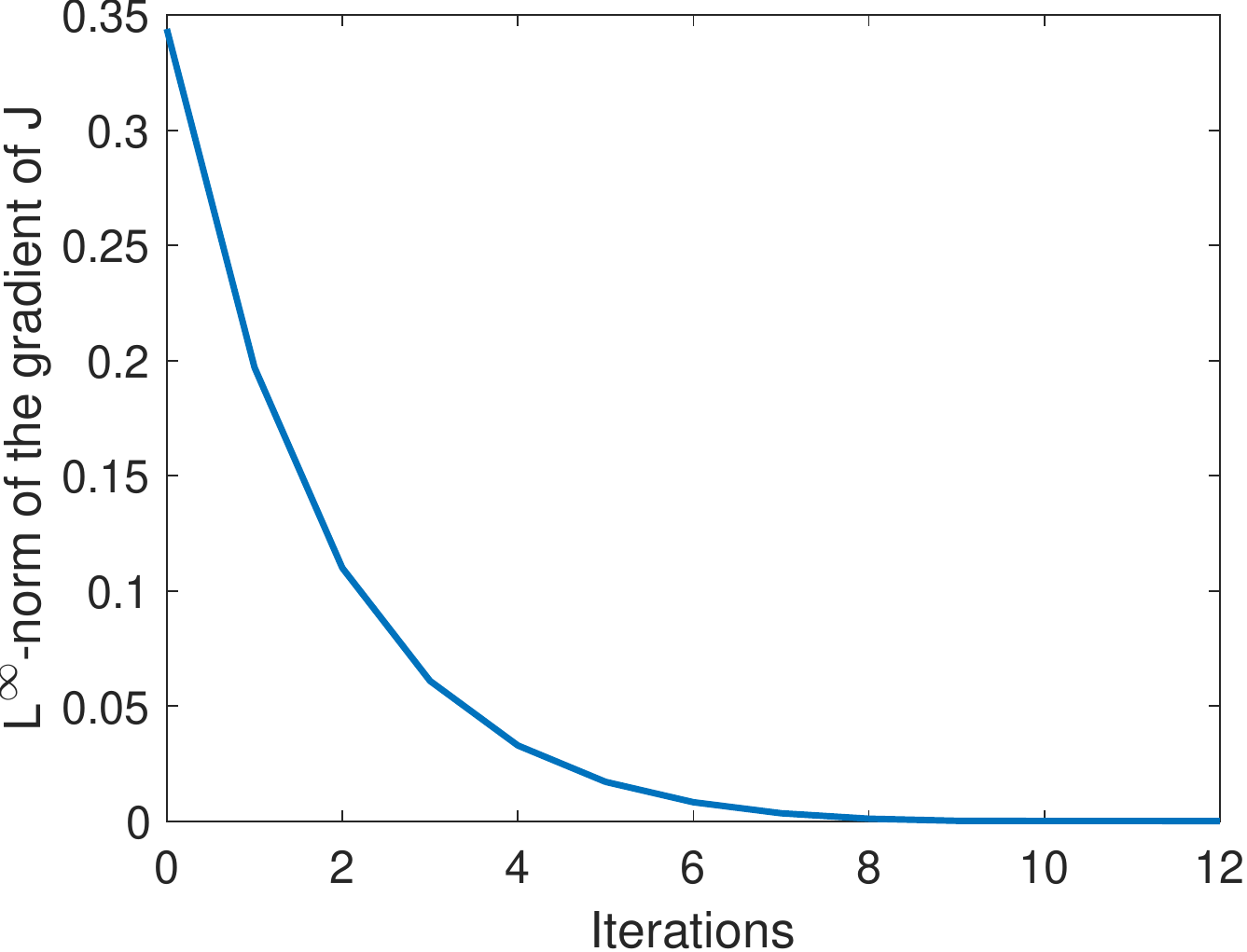}
\end{tabular}
 \caption{Simulation results for Example 1:  History of the cost function $J$ and the $L^\infty$-norm of $J^\prime$  in the case of  
  $\varepsilon=0.05$ and $\rho=0.00001$.}
  \label{geom3} 
\end{center}
 \end{figure}

 \subsubsection{Example 2}
 {\color{black} In this example,  the exact   Lam\'e  parameters (Figure \ref{geom4})  to  be  reconstructed   are given by 
 \[
\mu_e(x_1,x_2)=\sqrt{x_1^2+x_2^2}\quad \text{ and }\quad \lambda_e(x_1,x_2)=1.
\]
{We reconstruct $\mu$ and $\lambda$ by minimizing the functional \eqref{J_reg} in the space of piecewise constant functions on
the FEM mesh.} The choice $(\lambda_0,\mu_0)=(0.3, 0.5)$ was made as an initial guess. 

The resulting reconstructions  and the absolute errors are depicted in Figures \ref{geom5}-\ref{geom6}
and Figures \ref{geom7}-\ref{geom8}. The history  of cost function $J$  and the $L^\infty$-norm of $J^\prime$ in the course of the optimization process  are depicted in Figures \ref{geom9}-\ref{geom10}. 
The parameter $\lambda_e$  is well reconstructed, while  the parameter $\mu_e$  is less well  reconstructed.
}
\begin{figure}[ht!]
\begin{center}
\begin{tabular}{c c}
 \includegraphics[height=0.2\textheight]{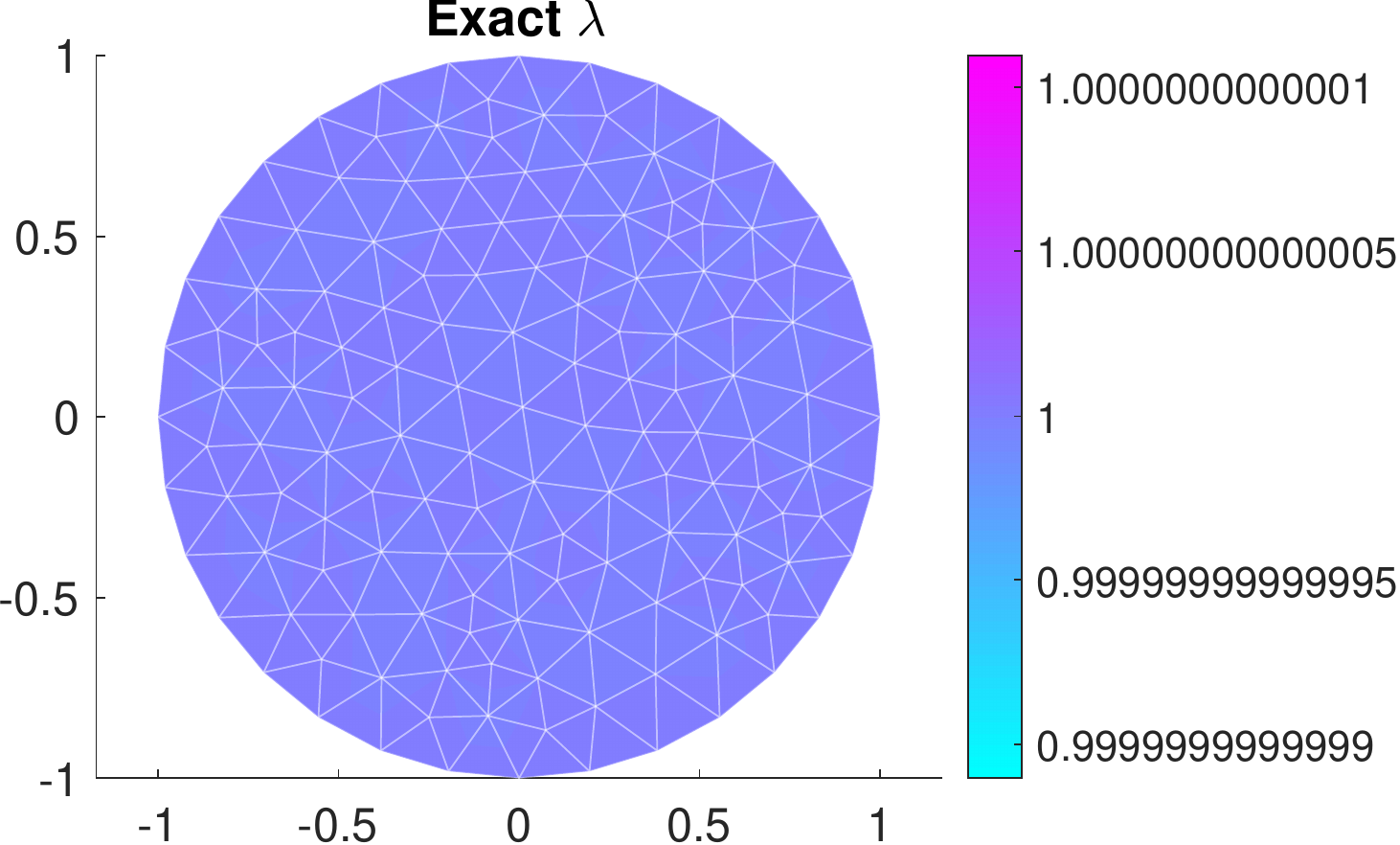} &
 \includegraphics[height=0.2\textheight]{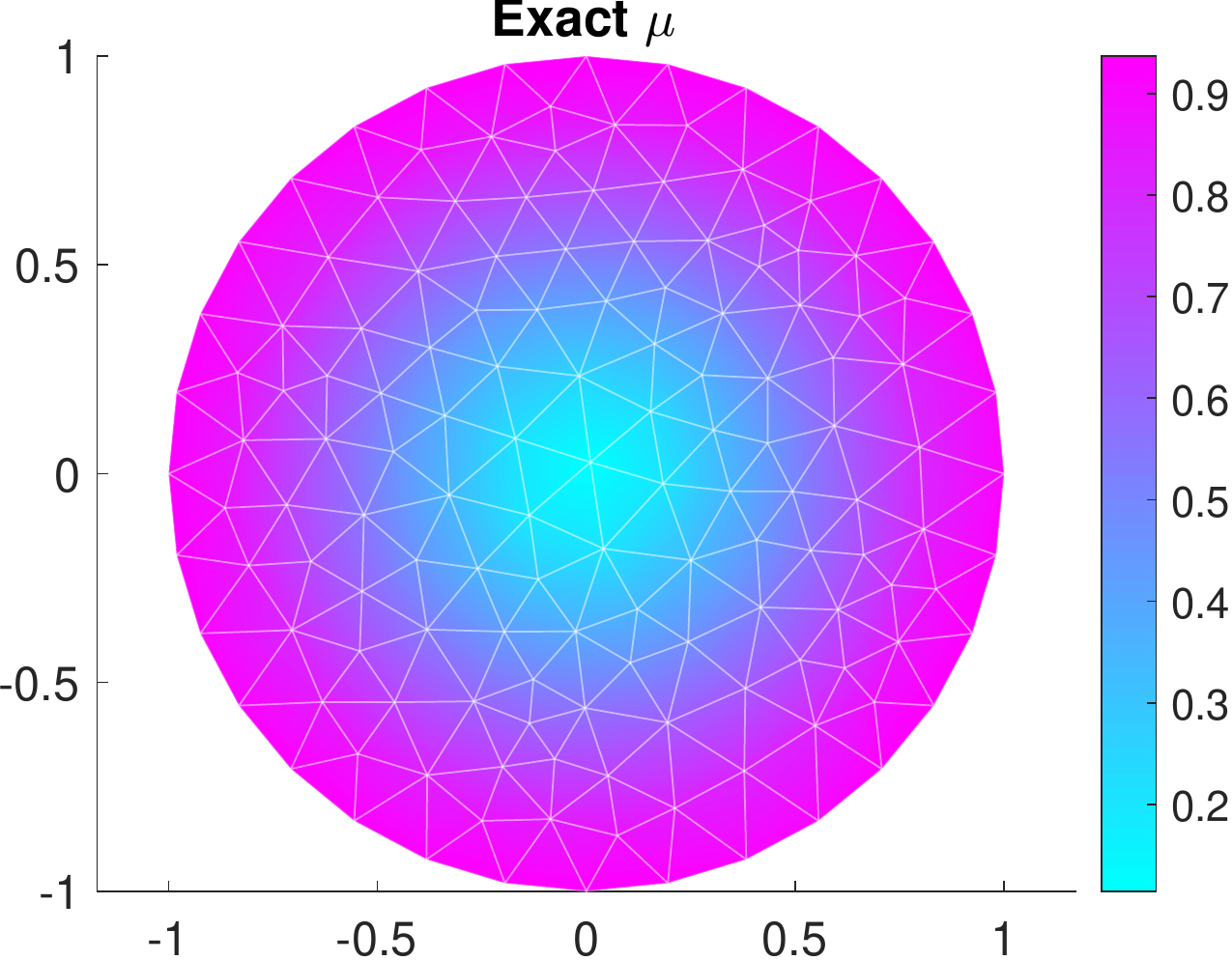}
\end{tabular}
 \caption{Simulation results for Example 2: The  exact Lam\'e  parameters.}
  \label{geom4}
  \end{center}
 \end{figure}
\begin{figure}[ht!]
\begin{center}
\begin{tabular}{c c} 
 \includegraphics[height=0.2\textheight]{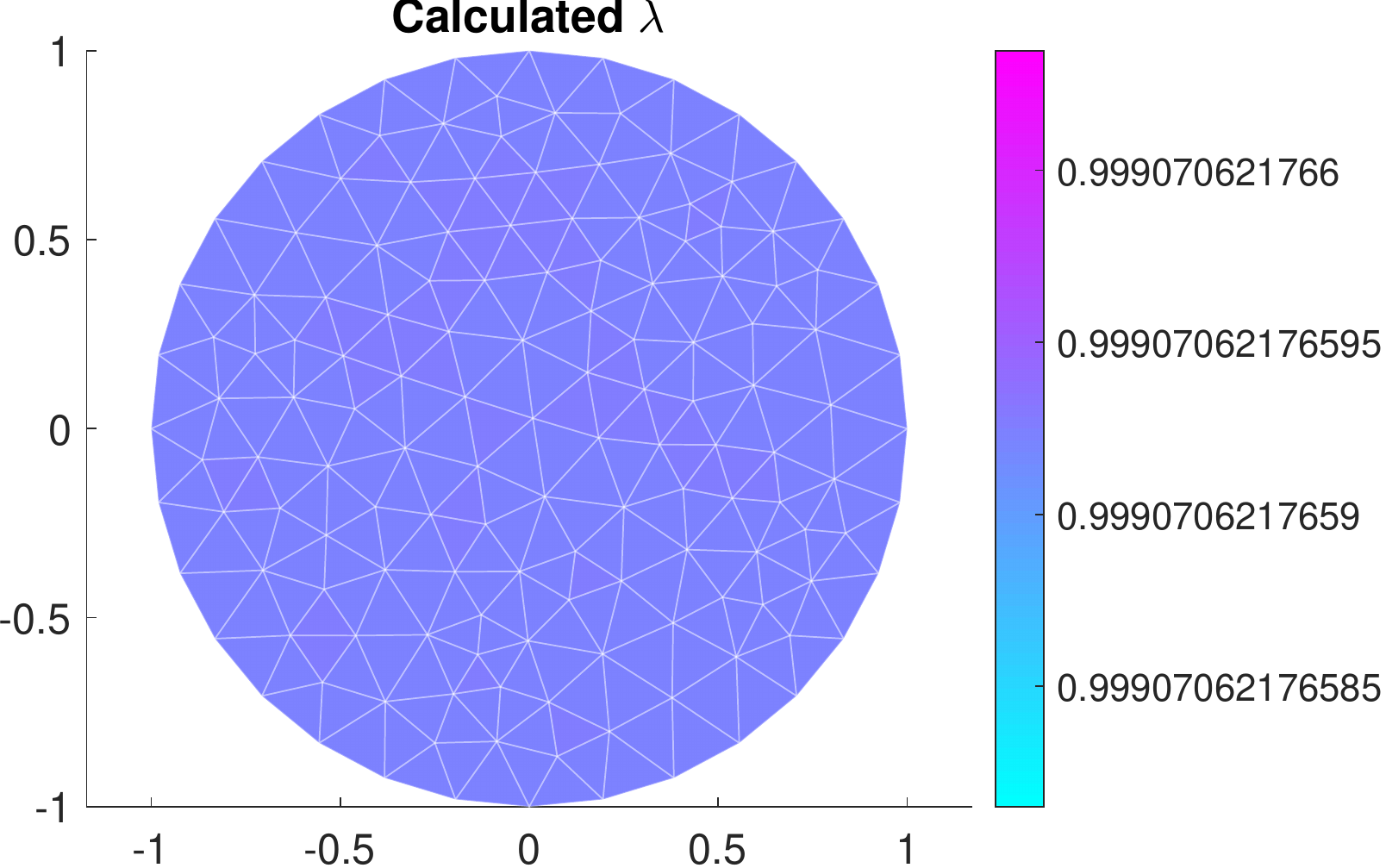} &
 \includegraphics[height=0.2\textheight]{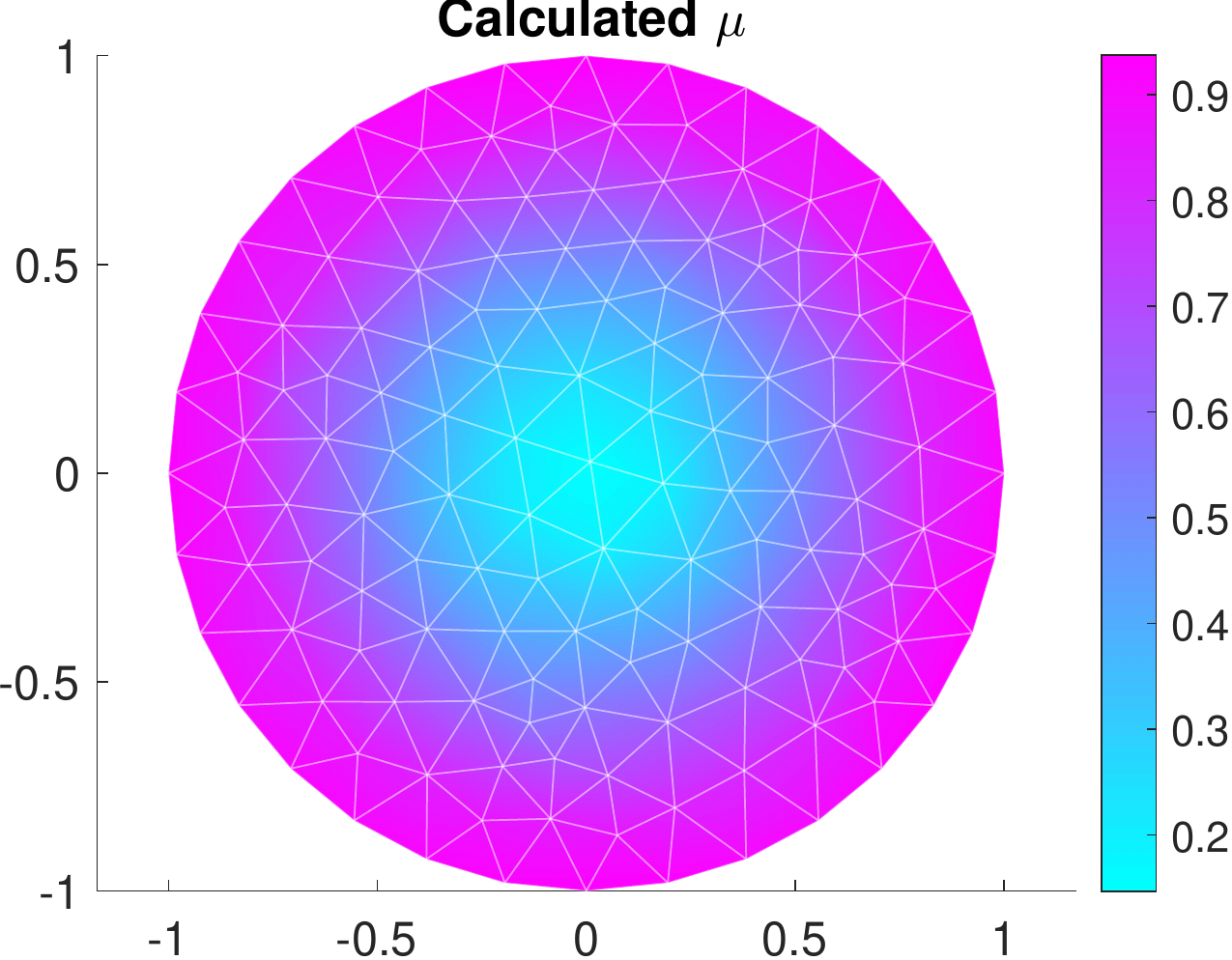}
 \end{tabular}
 \caption{Simulation results for Example 2: The computed  Lam\'e  parameters ($\epsilon=0.0$ and $\rho=0.0$).}
  \label{geom5}
  \end{center}
 \end{figure}
\begin{figure}[ht!]
\begin{center}
\begin{tabular}{c c}
 \includegraphics[height=0.2\textheight]{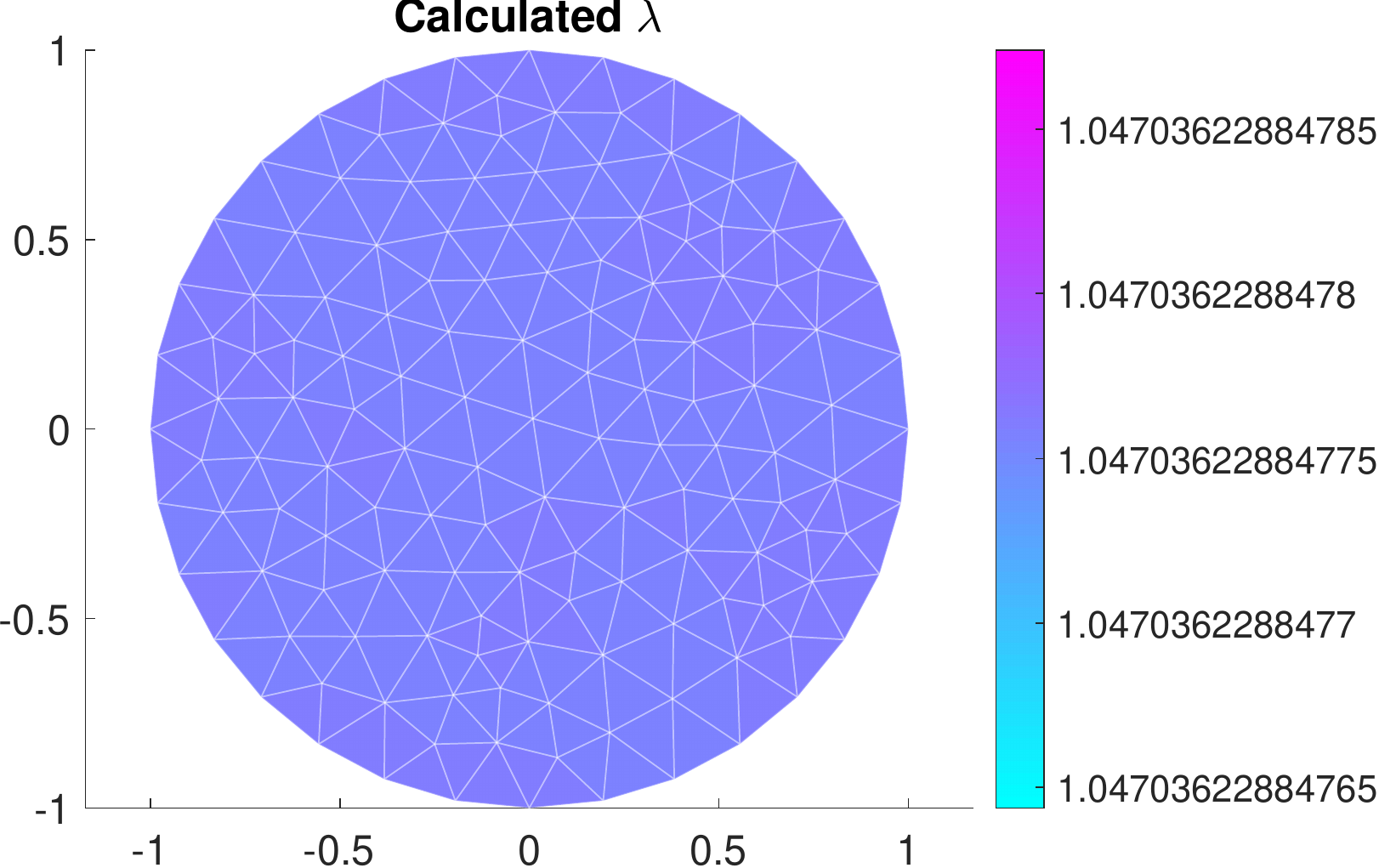} &
 \includegraphics[height=0.2\textheight]{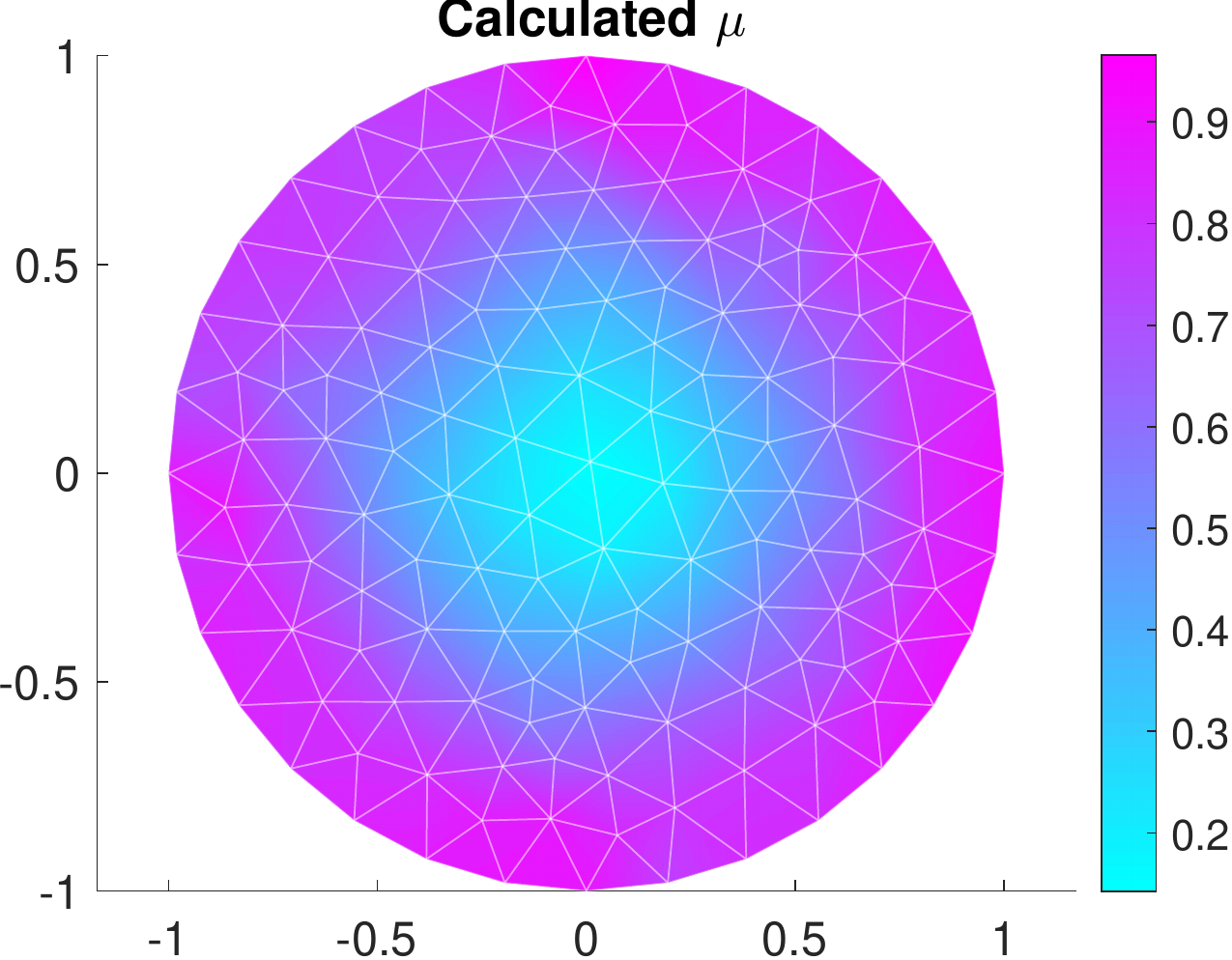}
\end{tabular}
 \caption{Simulation results for Example 2: The computed  Lam\'e  parameters ($\epsilon=0.03$ and $\rho=0.0001$).}
  \label{geom6}
  \end{center}
 \end{figure}
 \begin{figure}[ht!]
 \begin{center}
 \begin{tabular}{c c}
 \includegraphics[height=0.2\textheight]{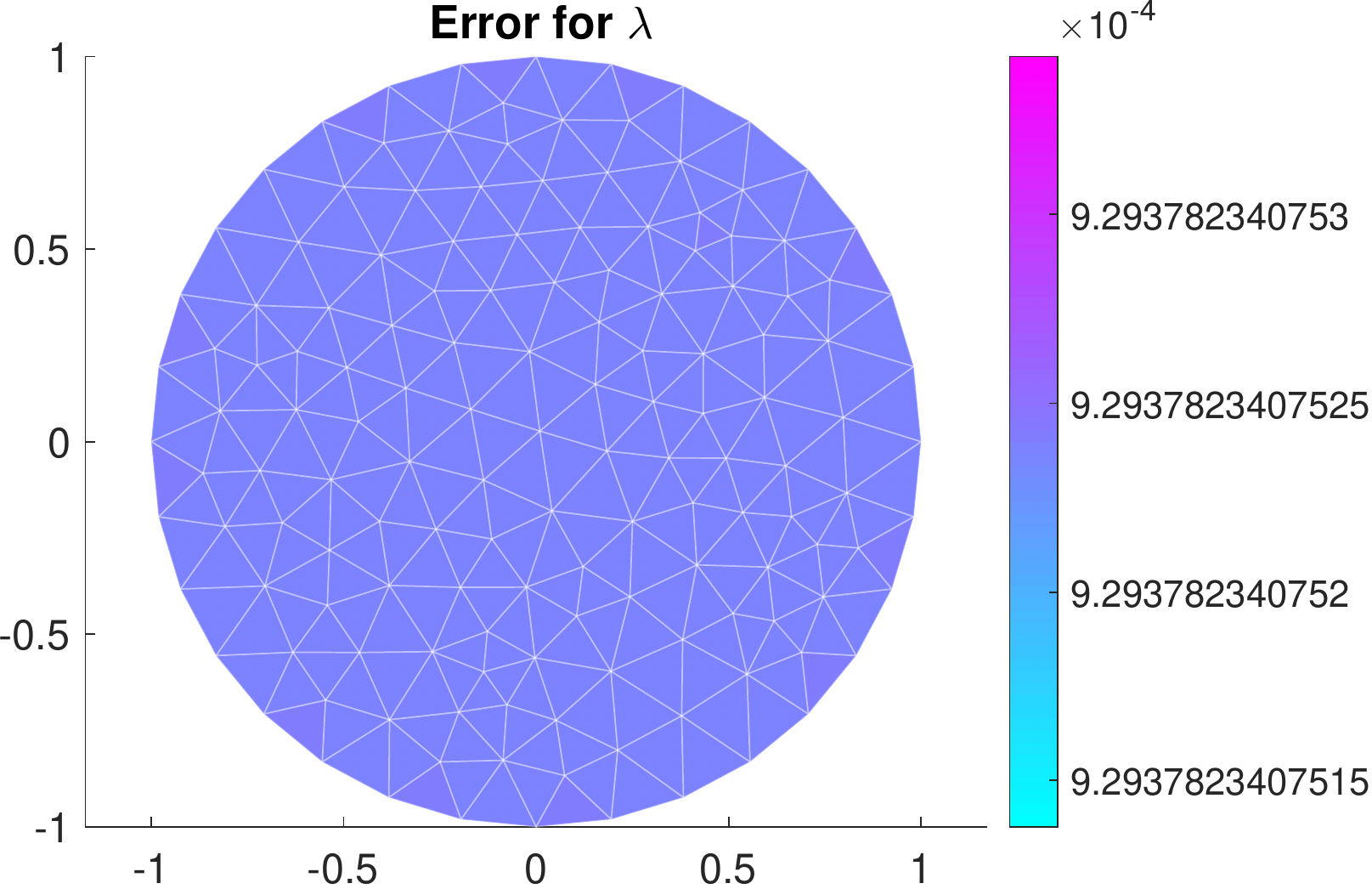} &
 \includegraphics[height=0.2\textheight]{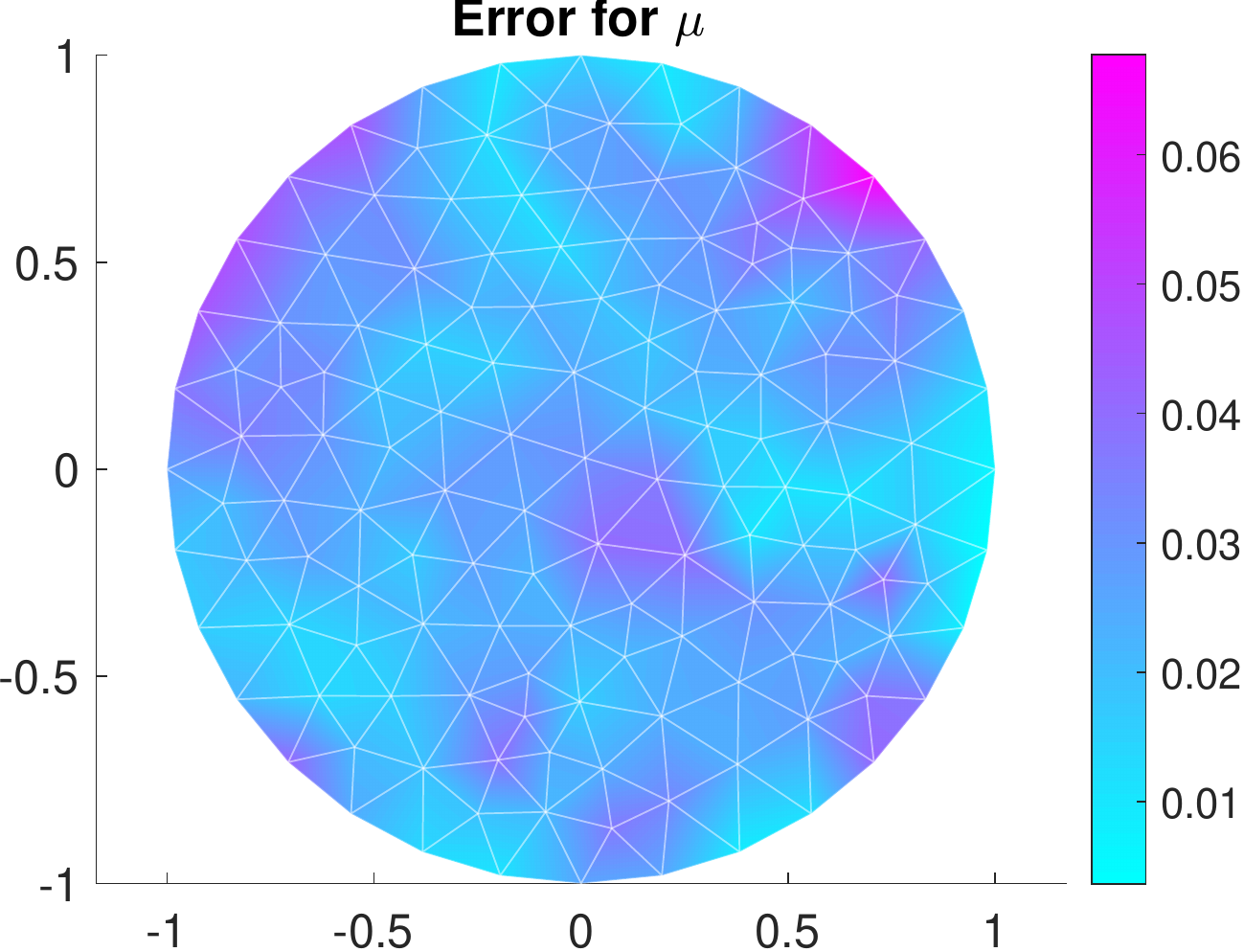}
 \end{tabular}
 \caption{Simulation results for Example 2:   Error for $\lambda$ and $\mu$ ($\epsilon=0.0$ and $\rho=0.0$).}
  \label{geom7}
  \end{center}
 \end{figure}
 
 \begin{figure}[ht!]
\begin{center}
\begin{tabular}{c c}
 \includegraphics[height=0.2\textheight]{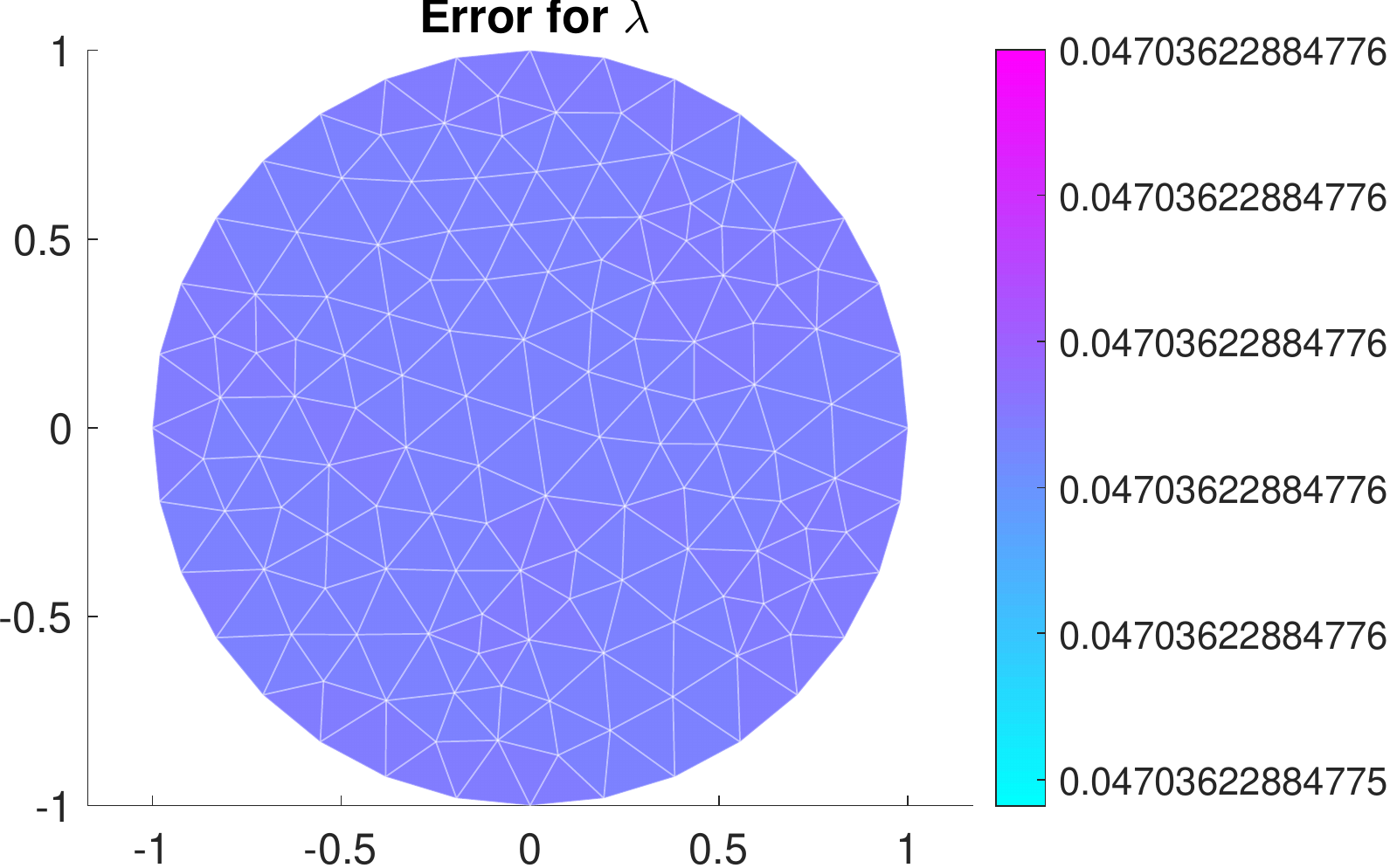} &
 \includegraphics[height=0.2\textheight]{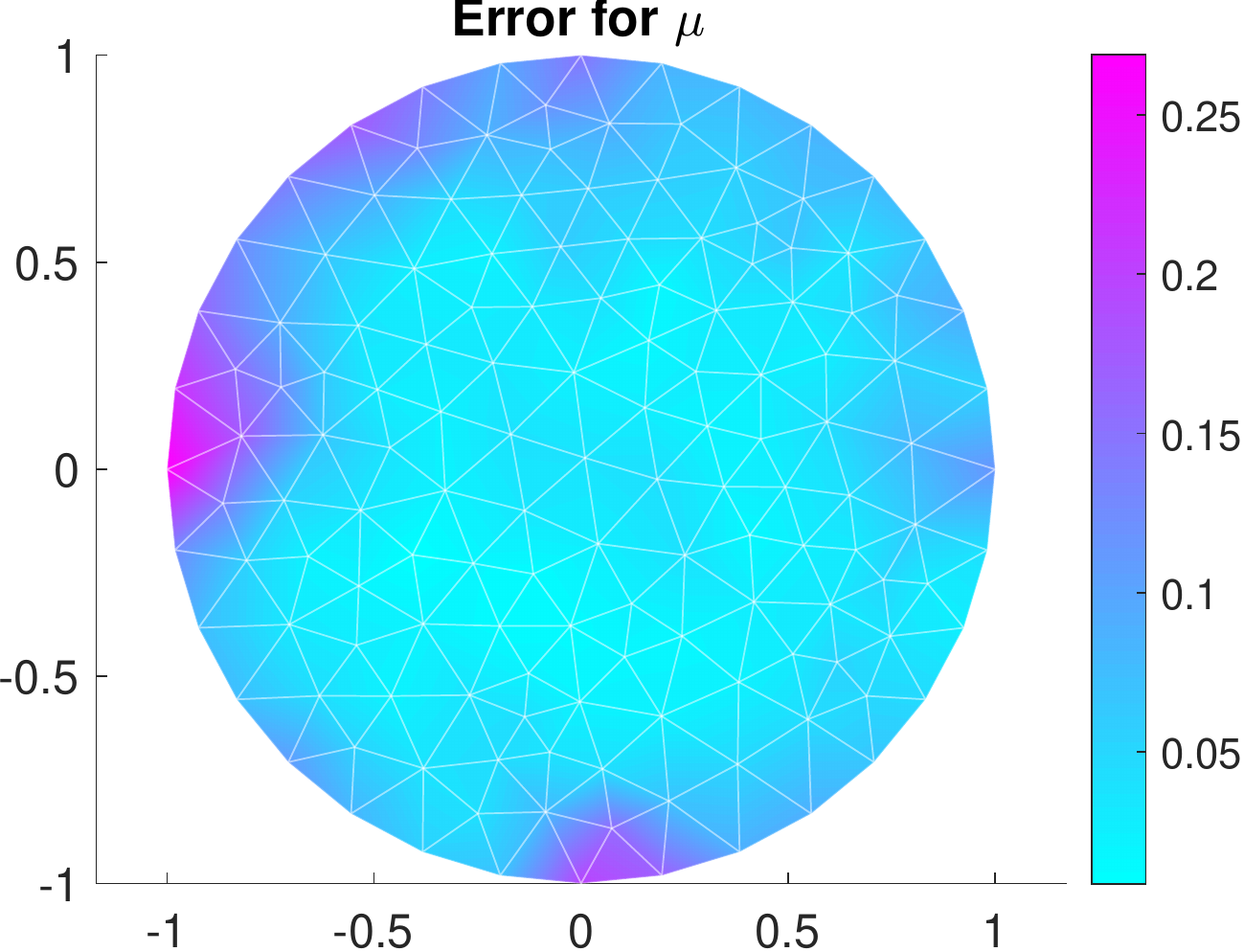}
\end{tabular}
 \caption{Simulation results for Example 2: Error for $\lambda$  and $\mu$ ($\epsilon=0.03$ and $\rho=0.0001$).}
  \label{geom8}
  \end{center}
 \end{figure}

 \begin{figure}[ht!]
 \begin{center}
 \begin{tabular}{c@{\qquad} c}
 \includegraphics[width=0.45\textwidth]{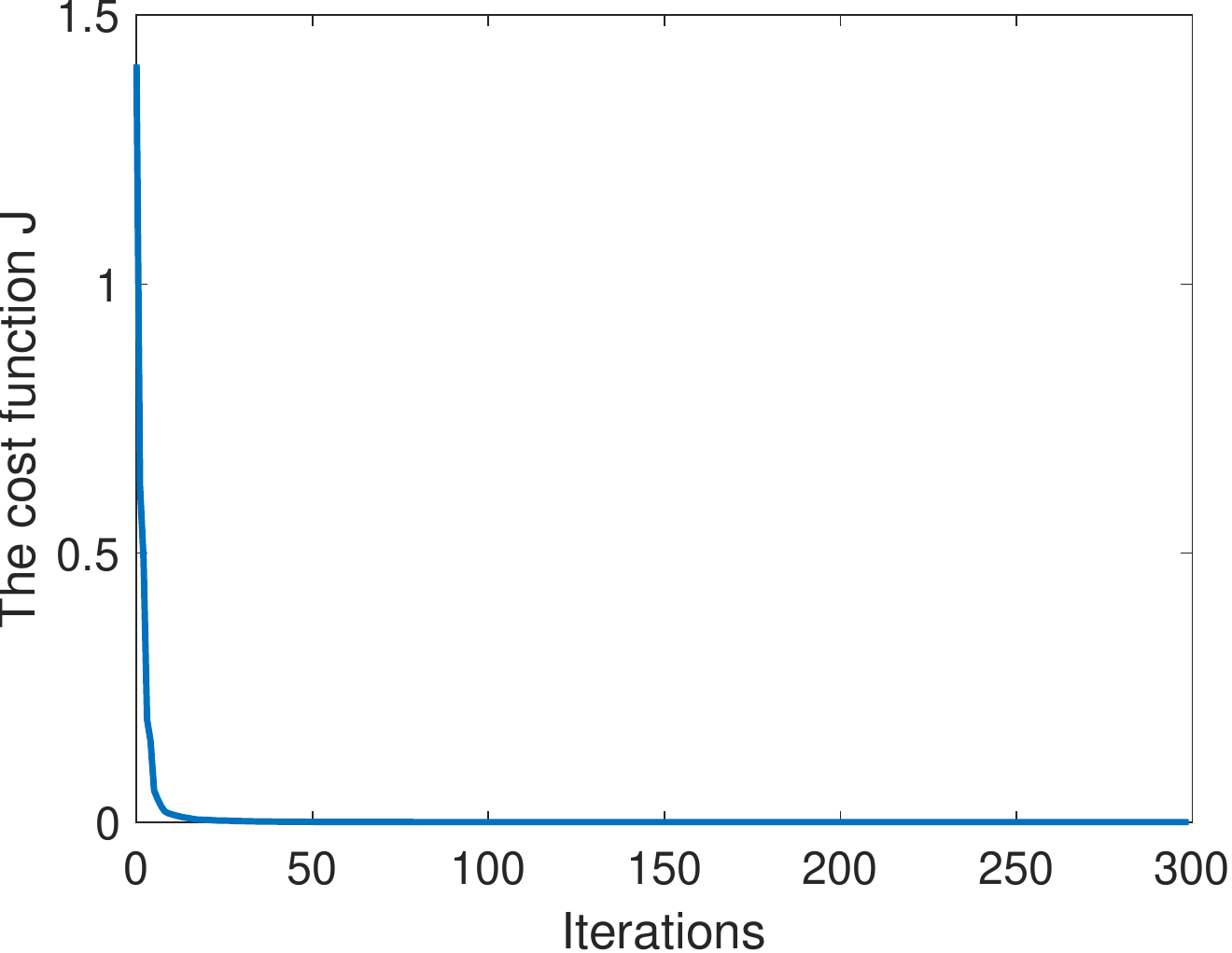} &
 \includegraphics[width=0.45\textwidth]{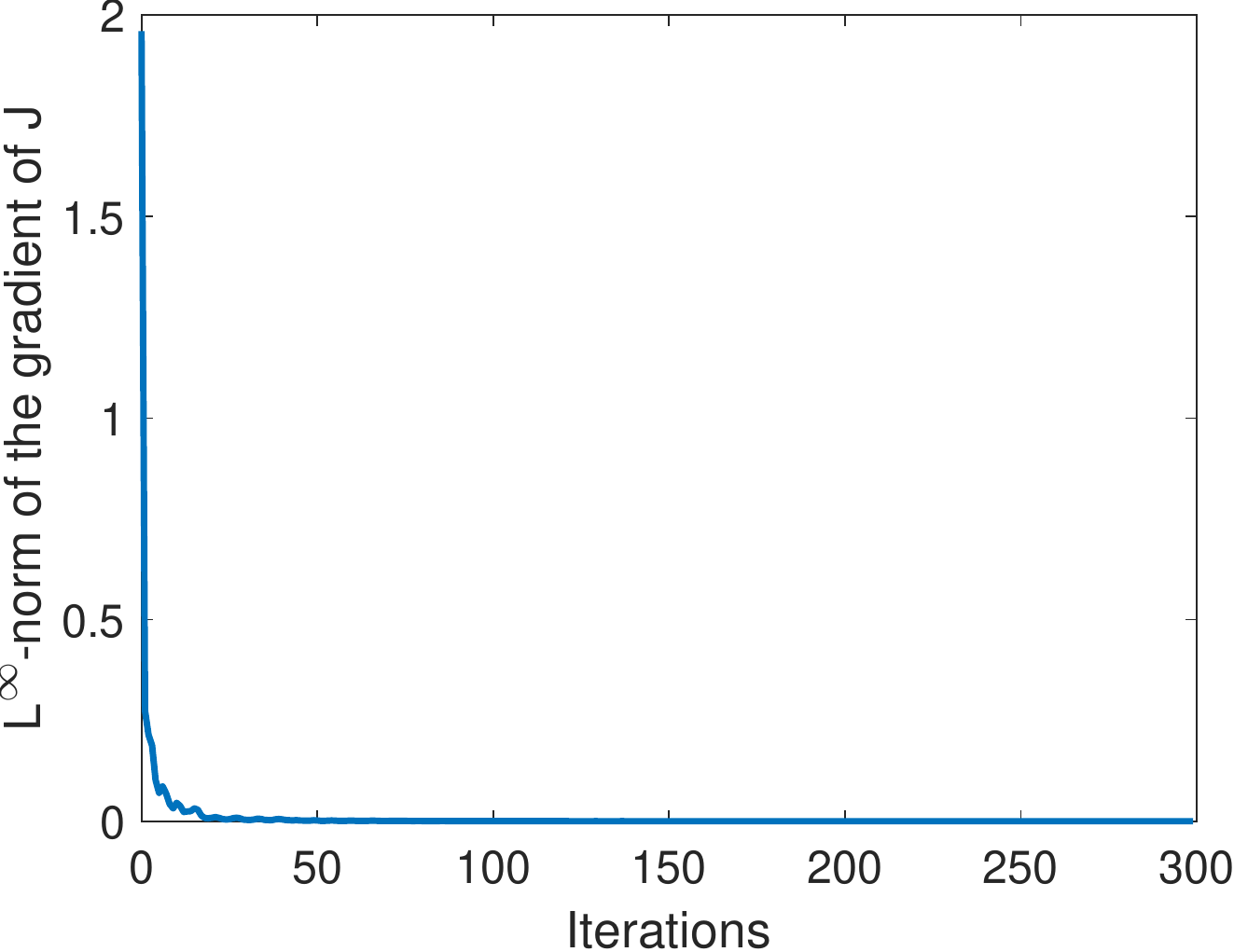}
\end{tabular}
 \caption{Simulation results for Example 2: History of the cost function $J$ and the $L^\infty$-norm of $J^\prime$ in the  $\epsilon=0.0$ and $\rho=0.0$.}
  \label{geom9}
  \end{center}
 \end{figure}
 \begin{figure}[ht!]
 \begin{center}
 \begin{tabular}{c@{\qquad} c}
 \includegraphics[width=0.45\textwidth]{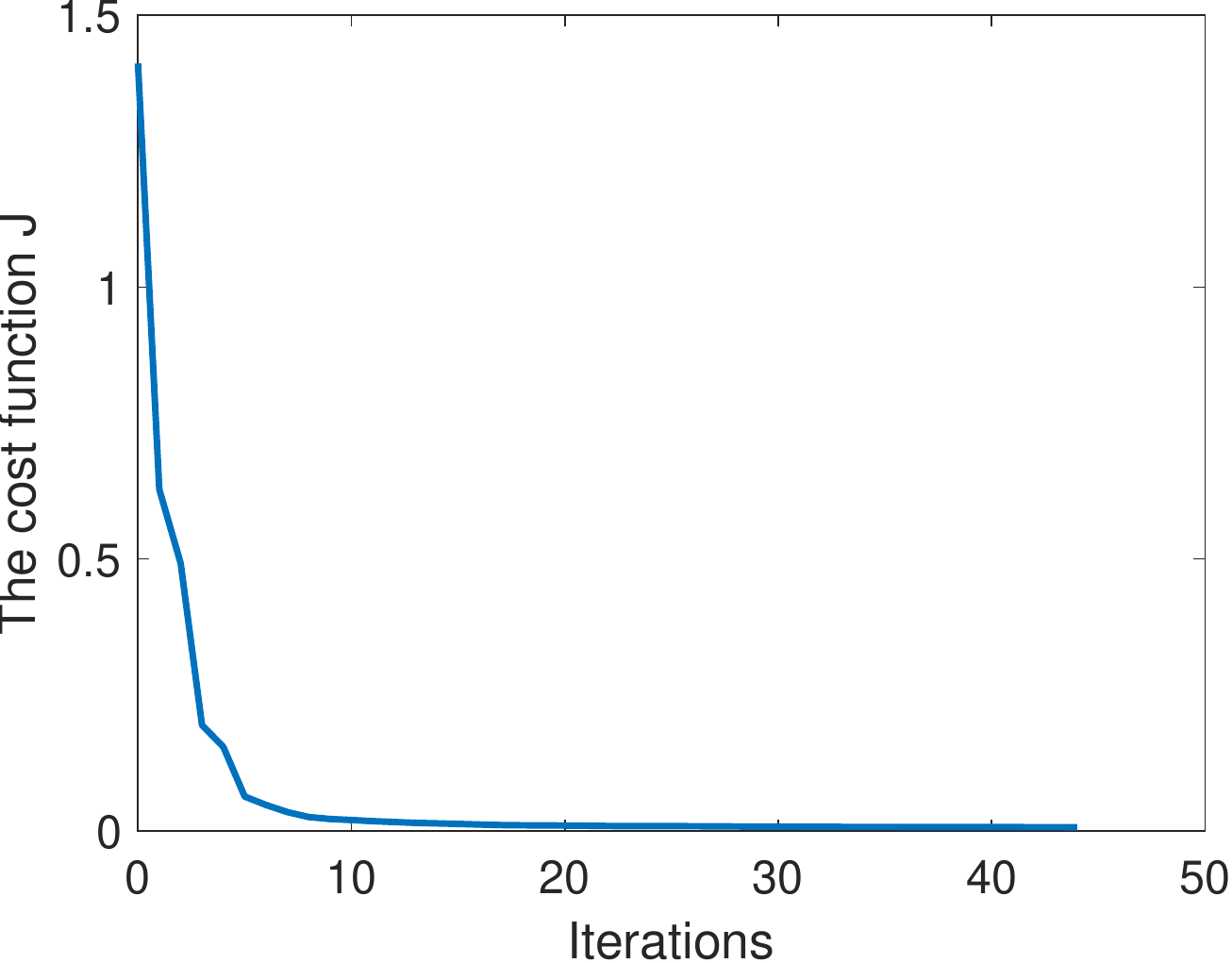} &
 \includegraphics[width=0.45\textwidth]{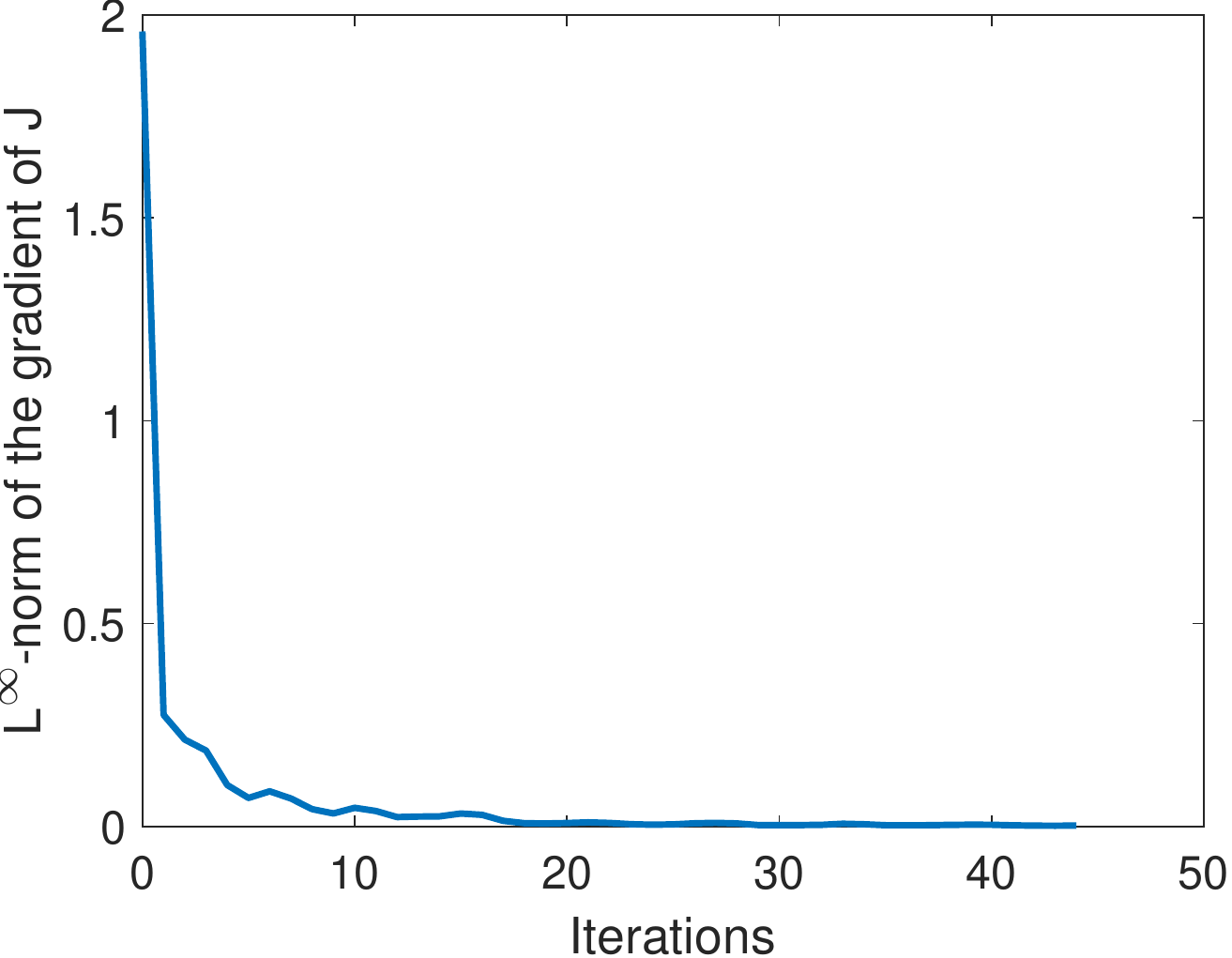}
 \end{tabular}
 \caption{Simulation results for Example 2: History of the cost function $J$ and the $L^\infty$-norm of $J^\prime$ in the  case of $\epsilon=0.03$ and $\rho=0.0001$.}
  \label{geom10}
  \end{center}
 \end{figure}

 \subsubsection{Example 3}
{\color{black} In this example, the exact Lam\'e parameters (Figure \ref{geom11}) to  be recovered  are  given by 
\[
 \mu_e=\sqrt{x_1^2+x_2^2},
 \]
 and 
\[
 \lambda_e=  \exp(-5((x_1-1/2)^2+( x_2-1/2)^2))+\exp(-5((x_1+1/2)^2+ (x_2+1/2)^2)).
\]
{As in Example 2, we reconstruct $\mu$ and $\lambda$ by minimizing the functional \eqref{J_reg} in the space of piecewise constant functions on
the FEM mesh, and use the initial guess $(\lambda_0,\mu_0)=(0.3, 0.5)$.}
The resulting reconstructions  and the absolute errors are depicted in Figures \ref{geom12}-\ref{geom13}
and Figures \ref{geom14}-\ref{geom15}. The history  of cost function $J$  and the $L^\infty$-norm of $J^\prime$ in the course of the optimization process  are depicted in Figures \ref{geom16}-\ref{geom17}. 
The parameter $\mu_e$  is well reconstructed, while  the parameter $\lambda_e$  is less well  reconstructed. However,  the
location  of  the peaks  are obtained.}
\begin{figure}[ht!]
\begin{center}
\begin{tabular}{c@{\qquad} c}
 \includegraphics[width=0.45\textwidth]{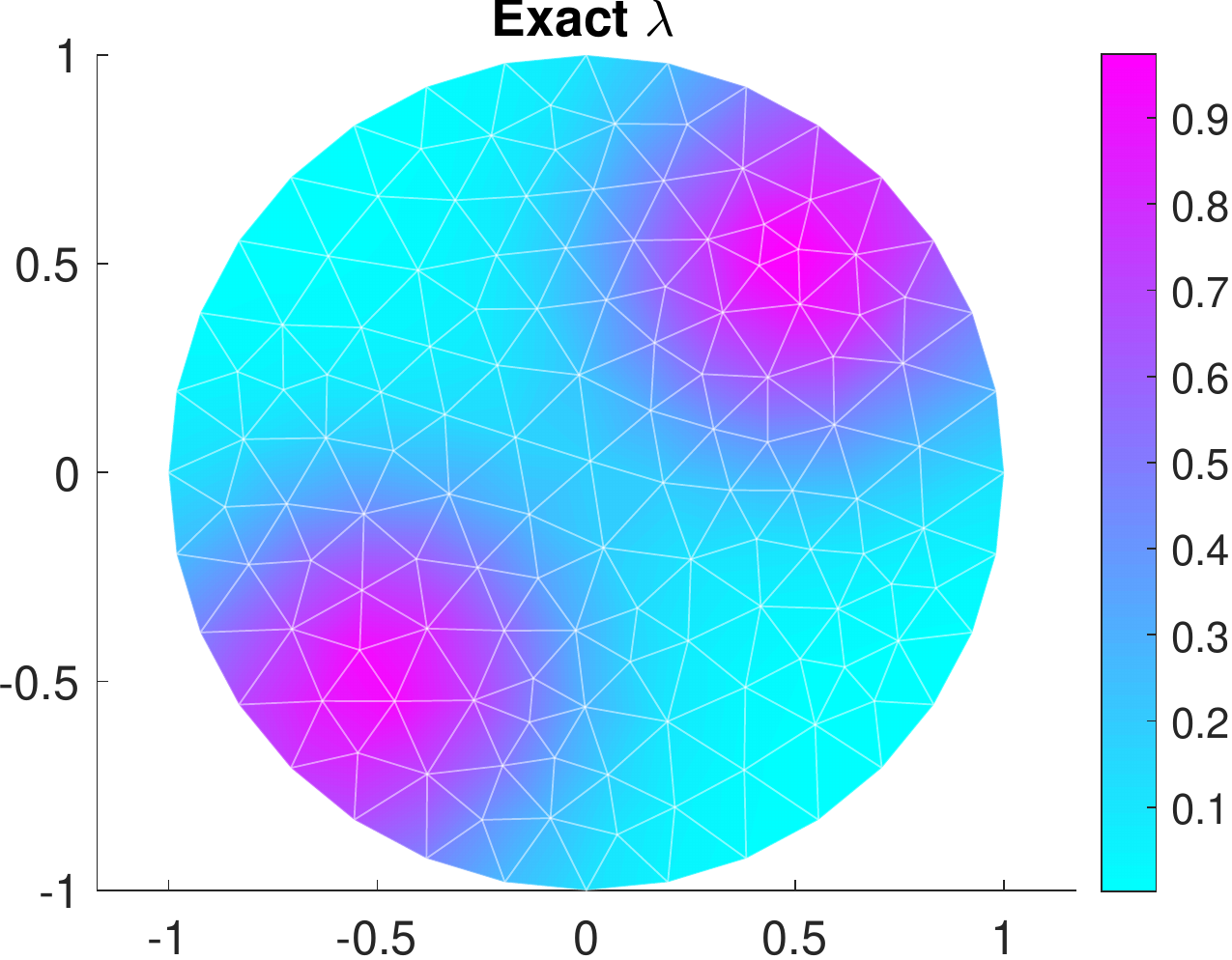} &
 \includegraphics[width=0.45\textwidth]{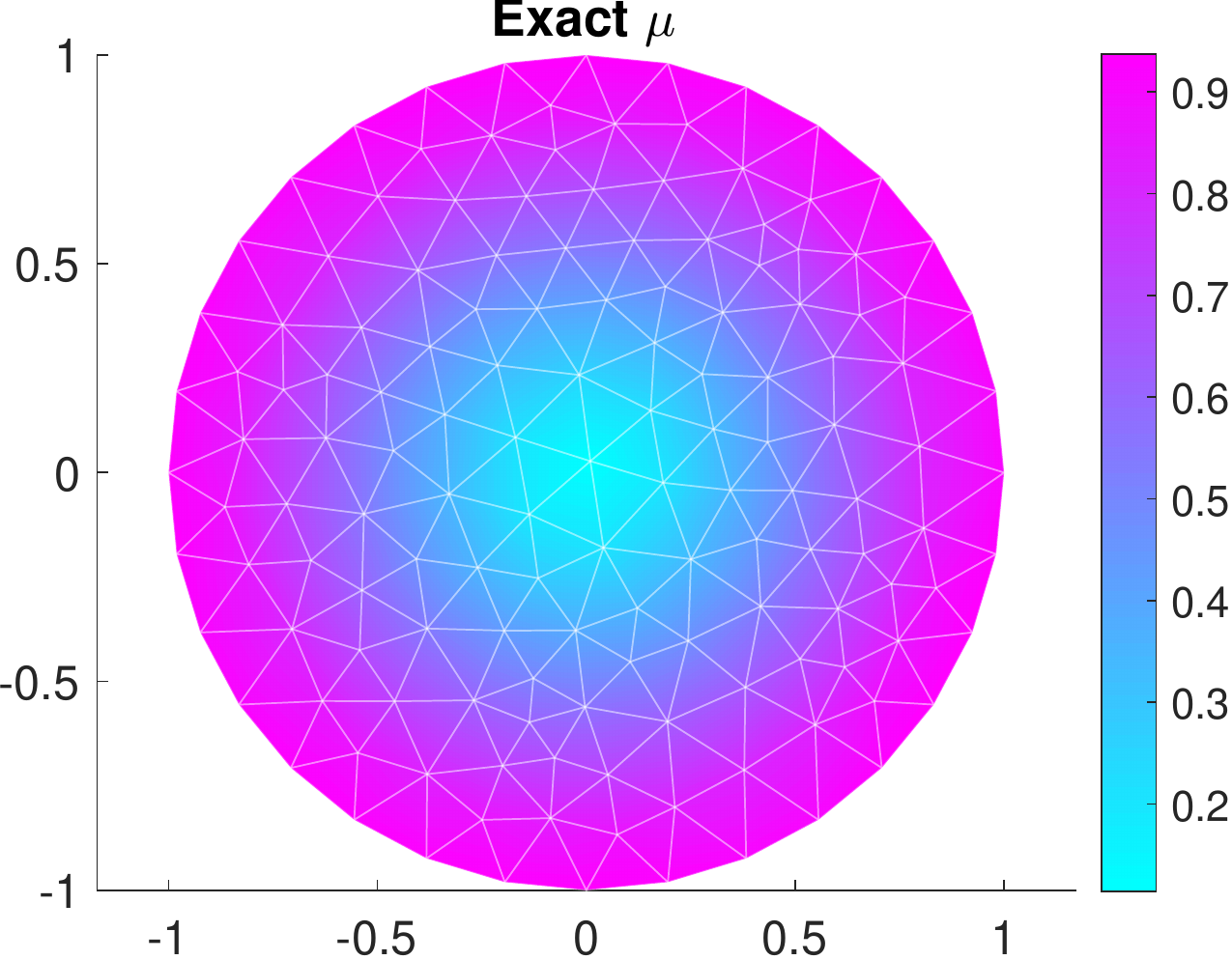}
\end{tabular}
 \caption{Simulation results for Example 3: The  exact Lam\'e  parameters.}
  \label{geom11}
  \end{center}
 \end{figure}
\begin{figure}[ht!]
\begin{center}
\begin{tabular}{c@{\qquad} c}
 \includegraphics[width=0.45\textwidth]{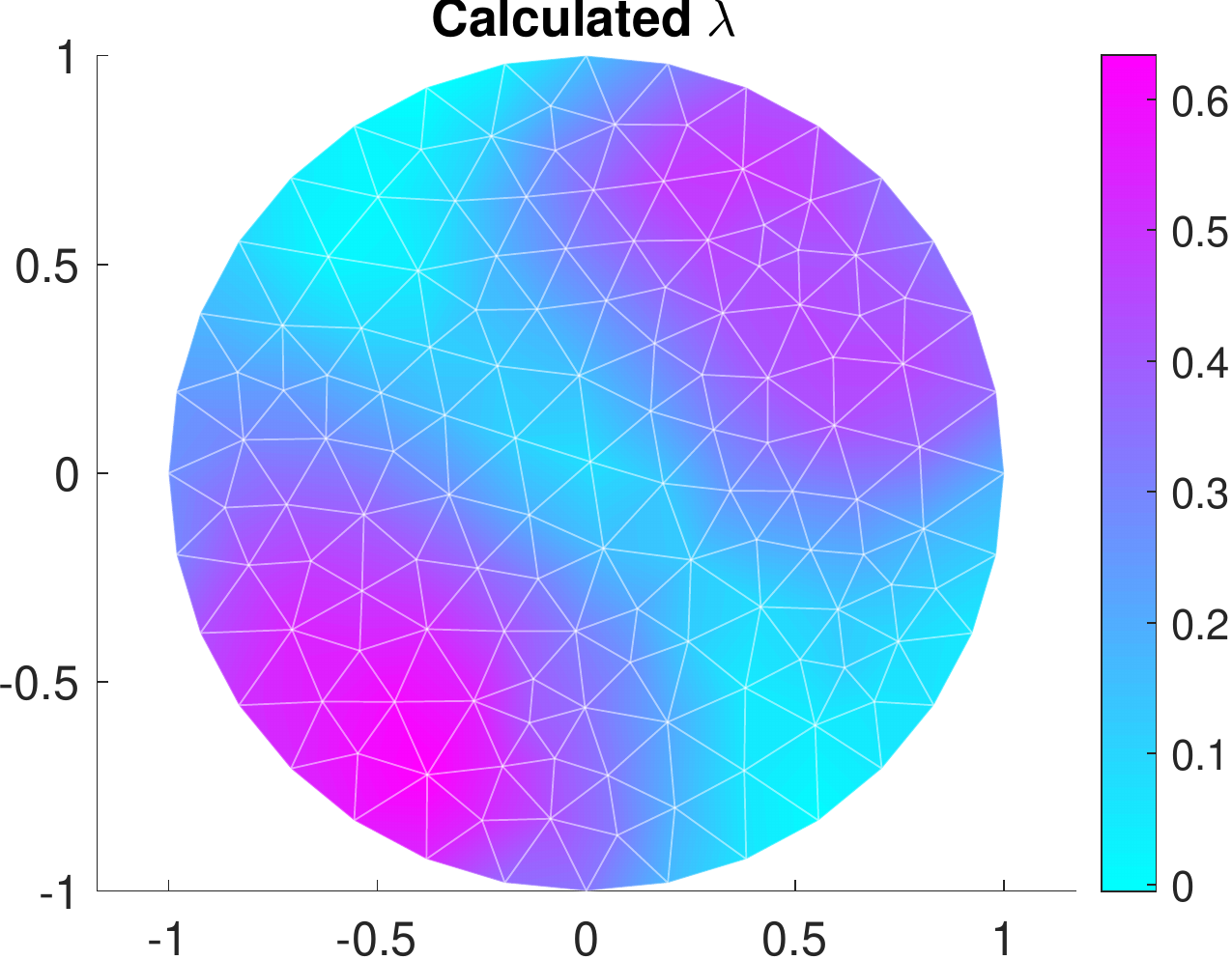} &
 \includegraphics[width=0.45\textwidth]{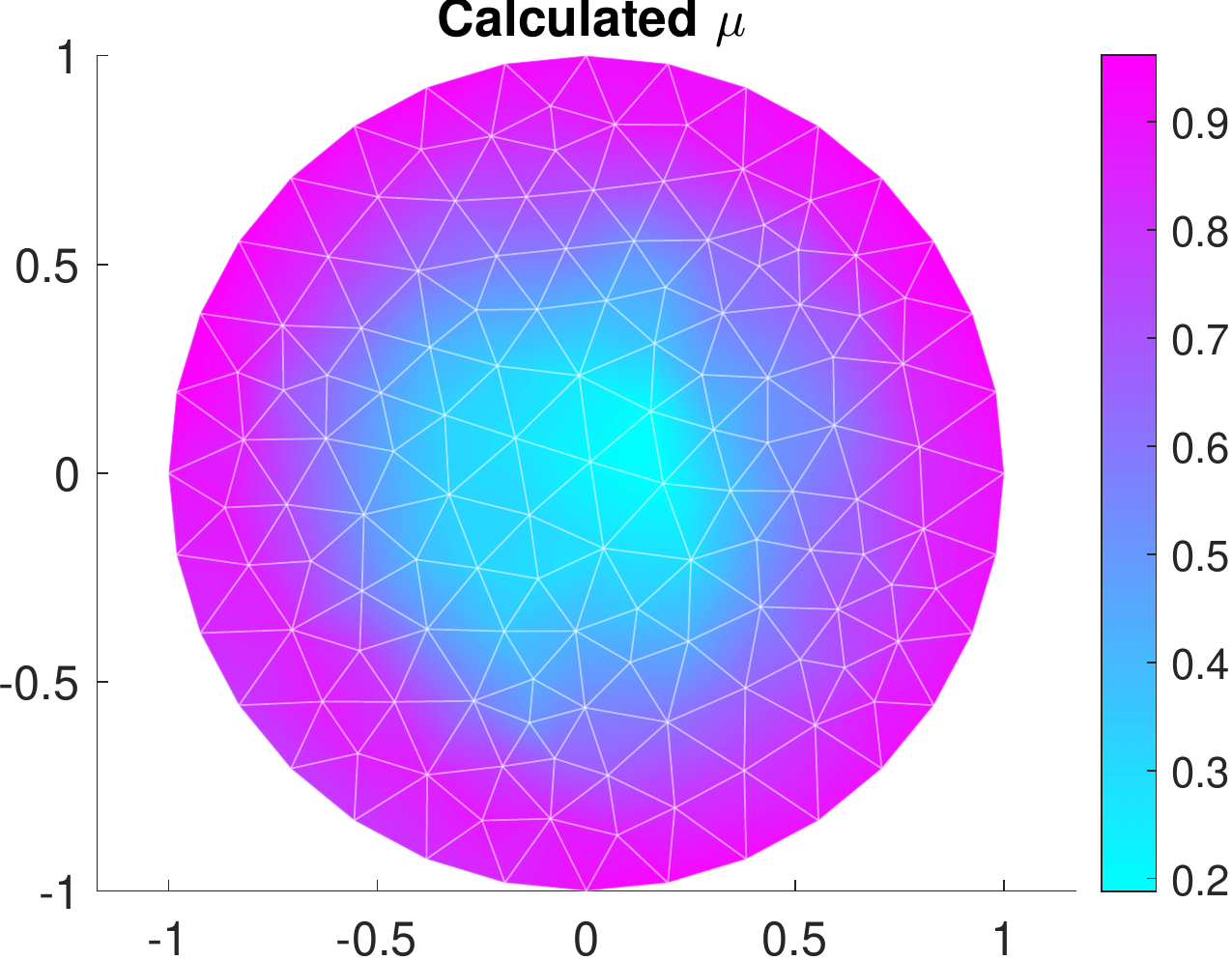}
 \end{tabular}
 \caption{Simulation results for Example 3: The computed  Lam\'e  parameters ($\epsilon=0.0$ and $\rho=0.0$).}
  \label{geom12}
  \end{center}
 \end{figure}
\begin{figure}[ht!]
\begin{center}
\begin{tabular}{c@{\qquad} c}
 \includegraphics[width=0.45\textwidth]{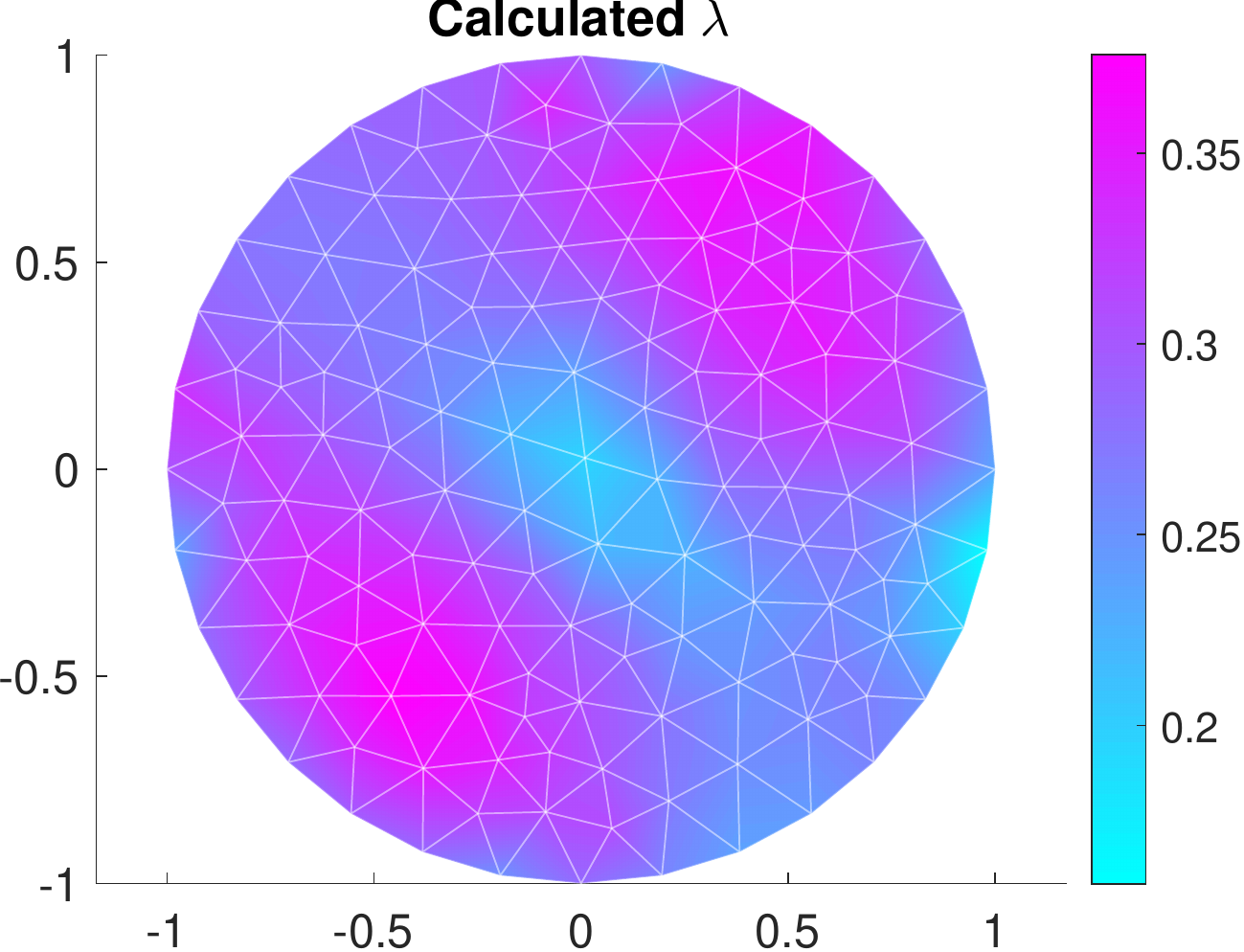} &
 \includegraphics[width=0.45\textwidth]{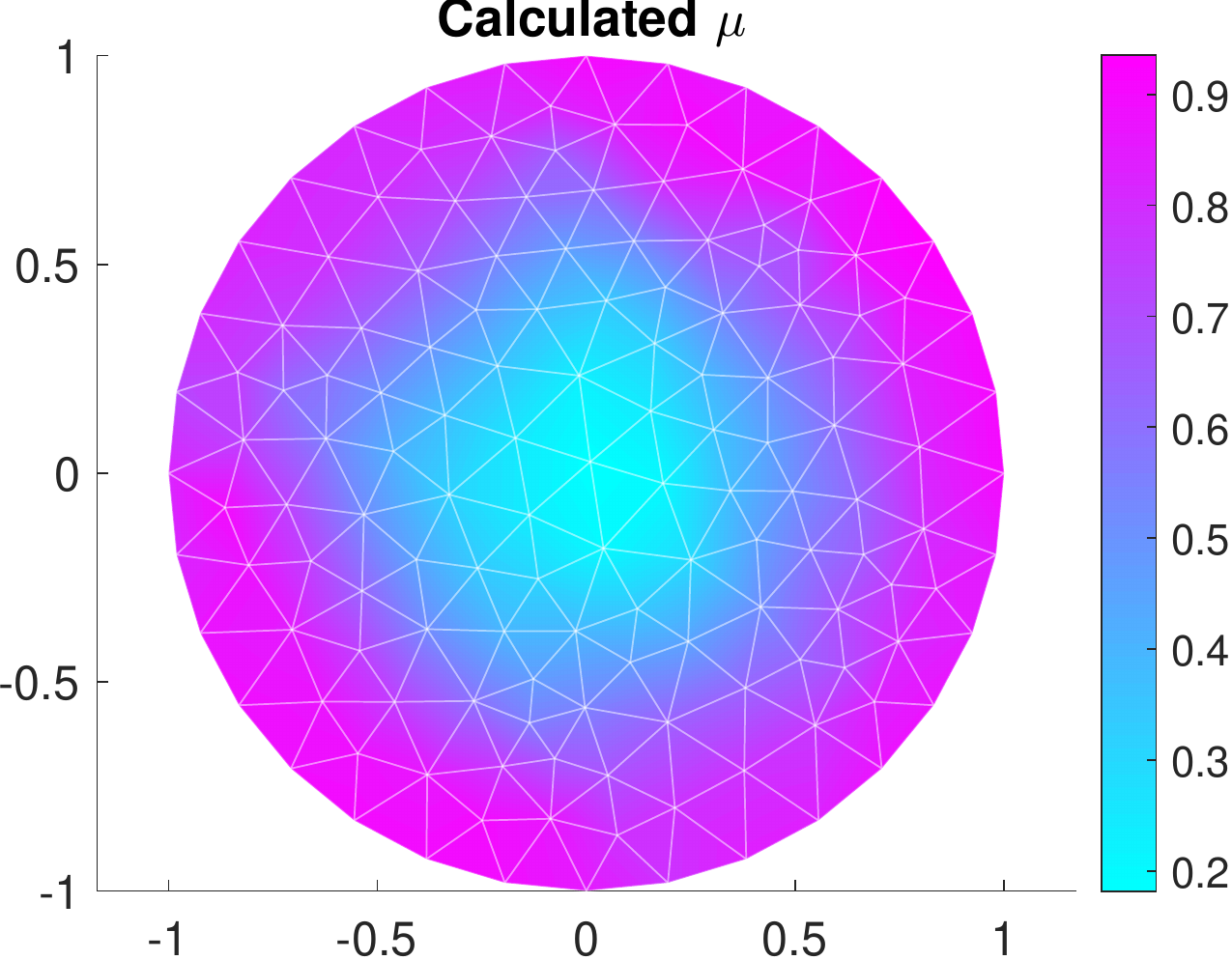}
 \end{tabular}
 \caption{Simulation results for Example 3: The computed  Lam\'e  parameters ($\epsilon=0.03$ and $\rho=0.0001$).}
  \label{geom13}
  \end{center}
 \end{figure}

 \begin{figure}[ht!]
 \begin{center}
 \begin{tabular}{c@{\qquad} c}
 \includegraphics[width=0.45\textwidth]{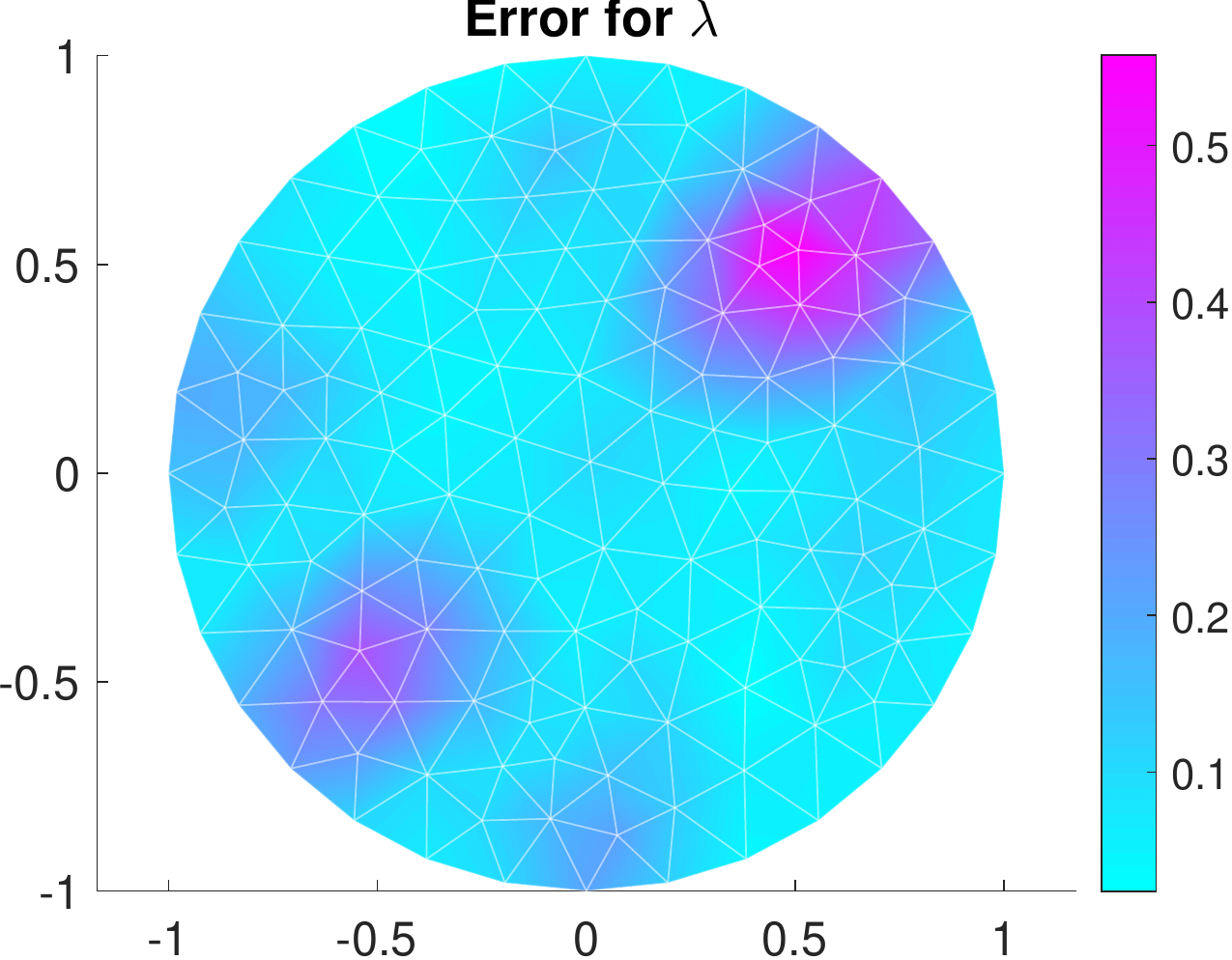} &
 \includegraphics[width=0.45\textwidth]{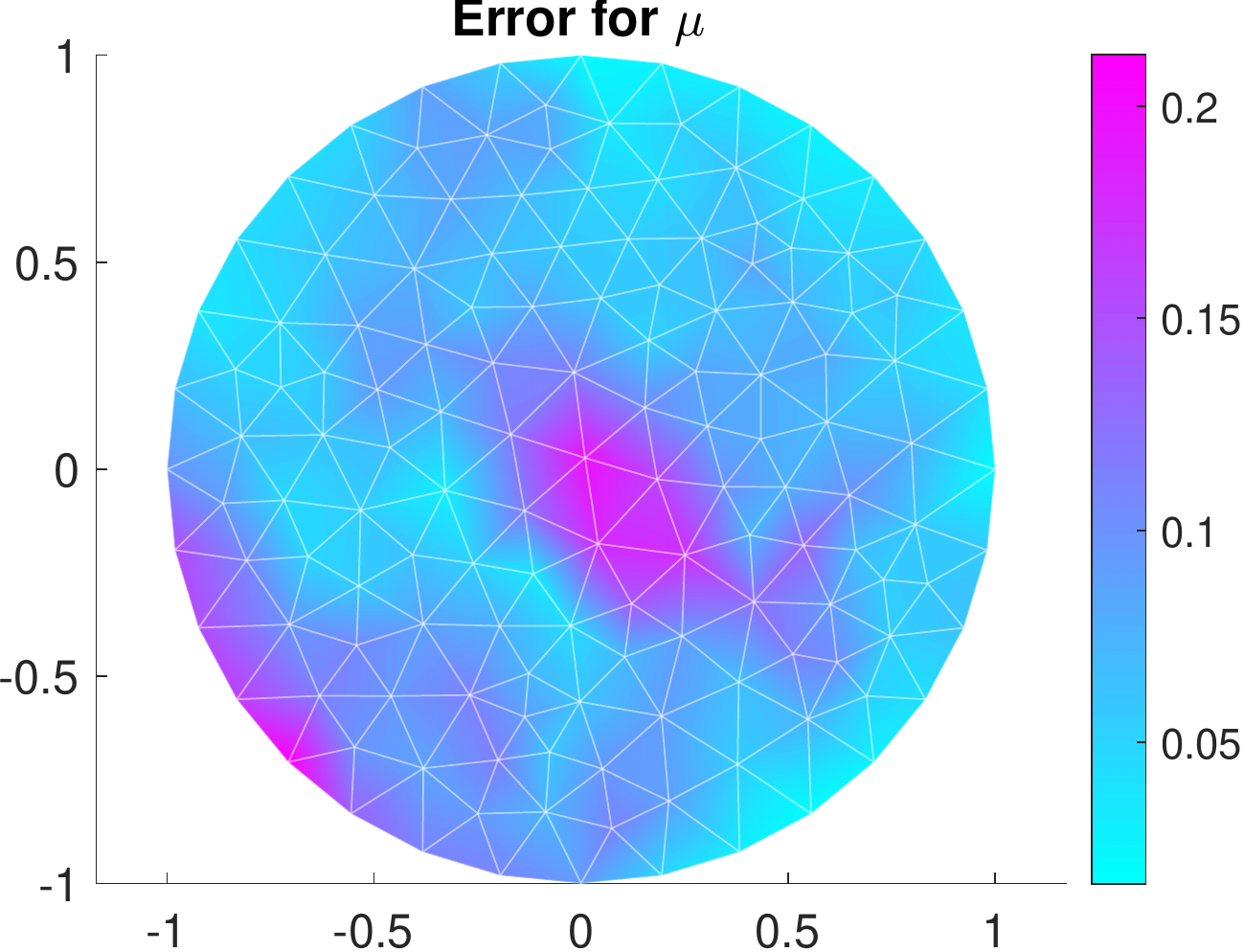}
 \end{tabular}
 \caption{Simulation results for Example 3:   Error for $\lambda$ and $\mu$ ($\epsilon=0.0$ and $\rho=0.0$).}
  \label{geom14}
  \end{center}
 \end{figure}
 \begin{figure}[ht!]
 \begin{center}
 \begin{tabular}{c@{\qquad} c}
 \includegraphics[width=0.45\textwidth]{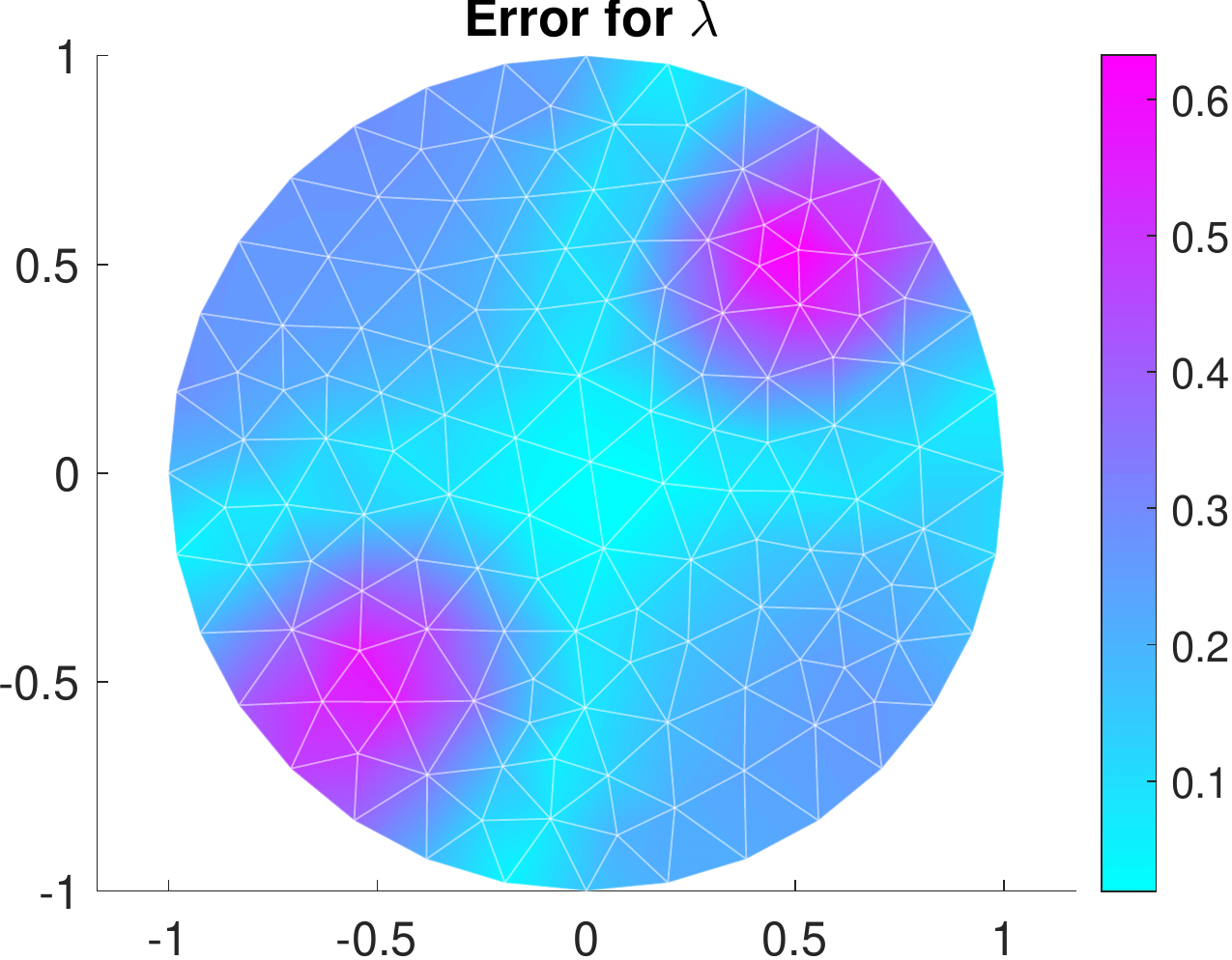} &
 \includegraphics[width=0.45\textwidth]{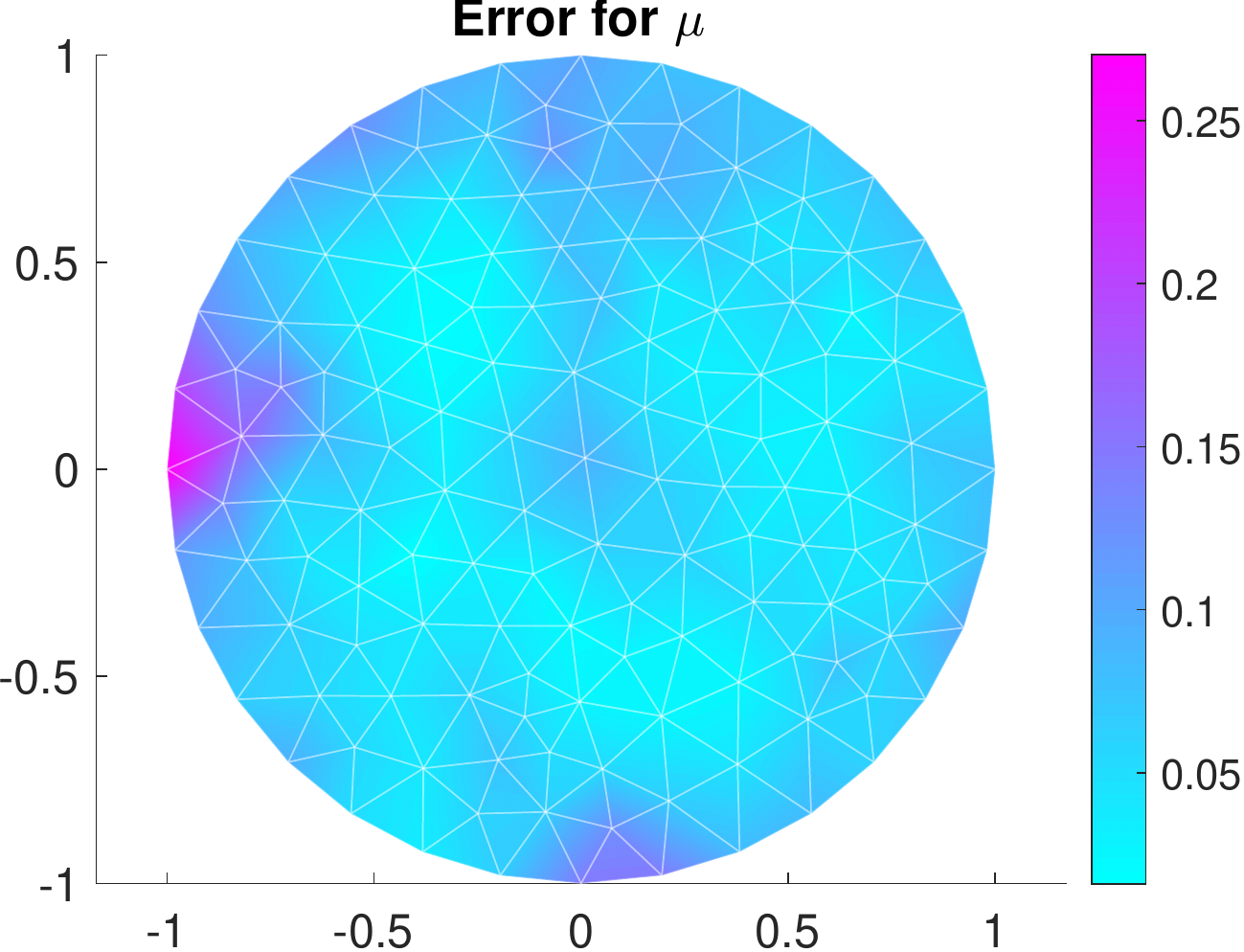}
 \end{tabular}
 \caption{Simulation results for Example 3:   Error for $\lambda$ and $\mu$ ($\epsilon=0.03$ and $\rho=0.0001$).}
  \label{geom15}
  \end{center}
 \end{figure}
 \begin{figure}[ht!]
 \begin{center}
 \begin{tabular}{c@{\qquad} c}
 \includegraphics[width=0.45\textwidth]{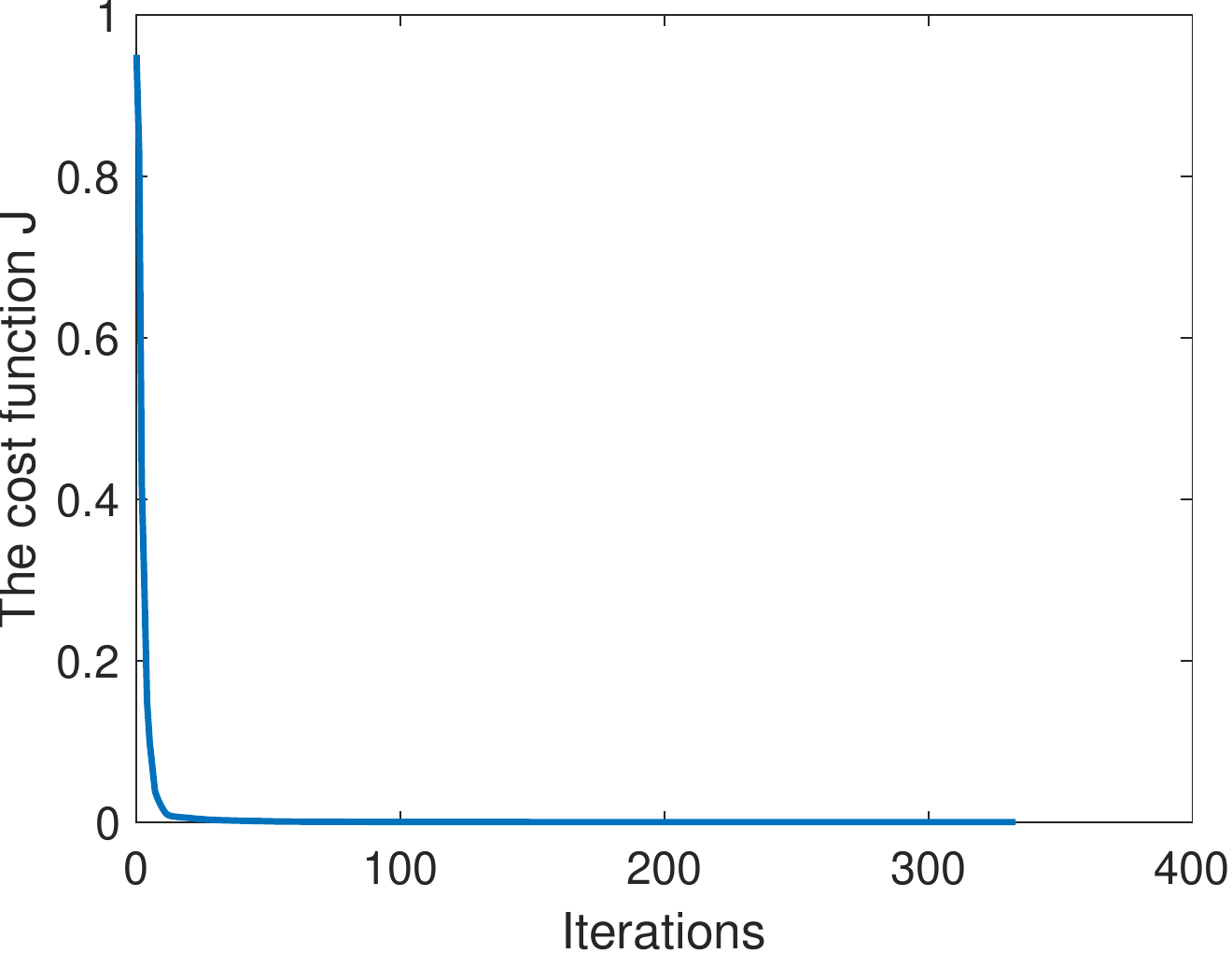} &
 \includegraphics[width=0.45\textwidth]{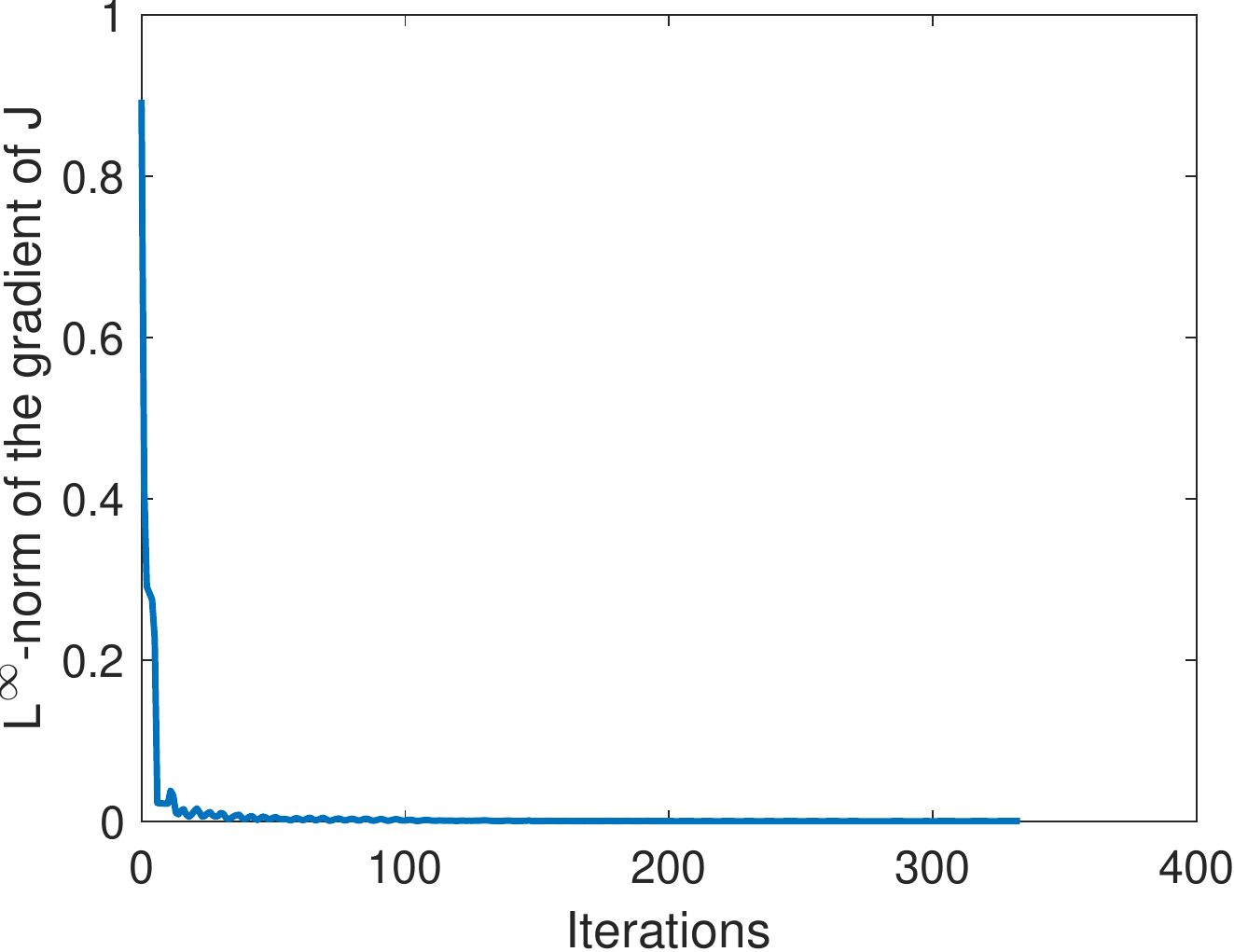}
\end{tabular} 
 \caption{Simulation results for Example 3: History of the cost function $J$ and the $L^\infty$-norm of $J^\prime$ ($\epsilon=0.0$ and $\rho=0.0$).}
  \label{geom16}
  \end{center}
 \end{figure}

  \begin{figure}[ht!]
 \begin{center}
 \begin{tabular}{c@{\qquad} c}
 \includegraphics[width=0.45\textwidth]{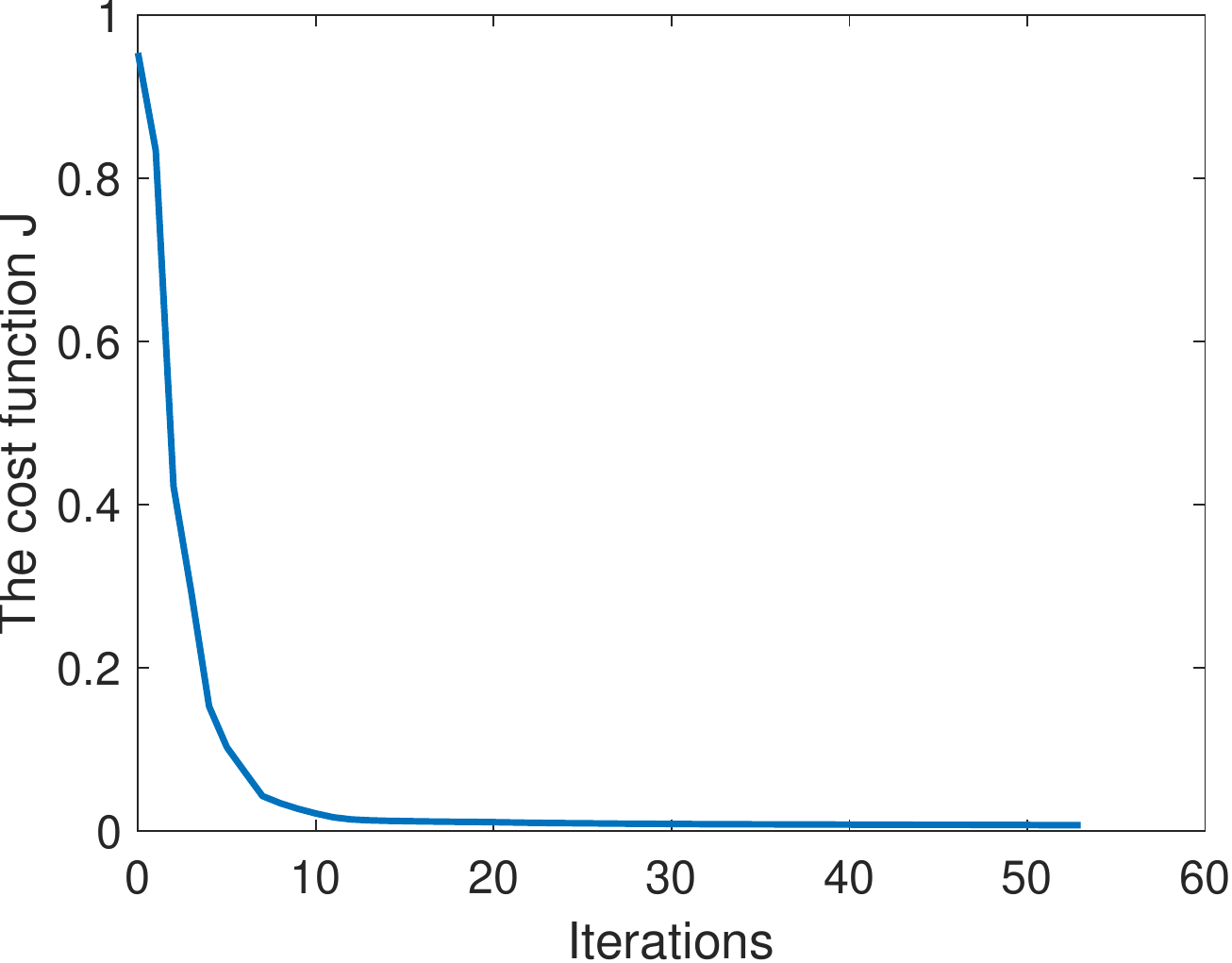} &
 \includegraphics[width=0.45\textwidth]{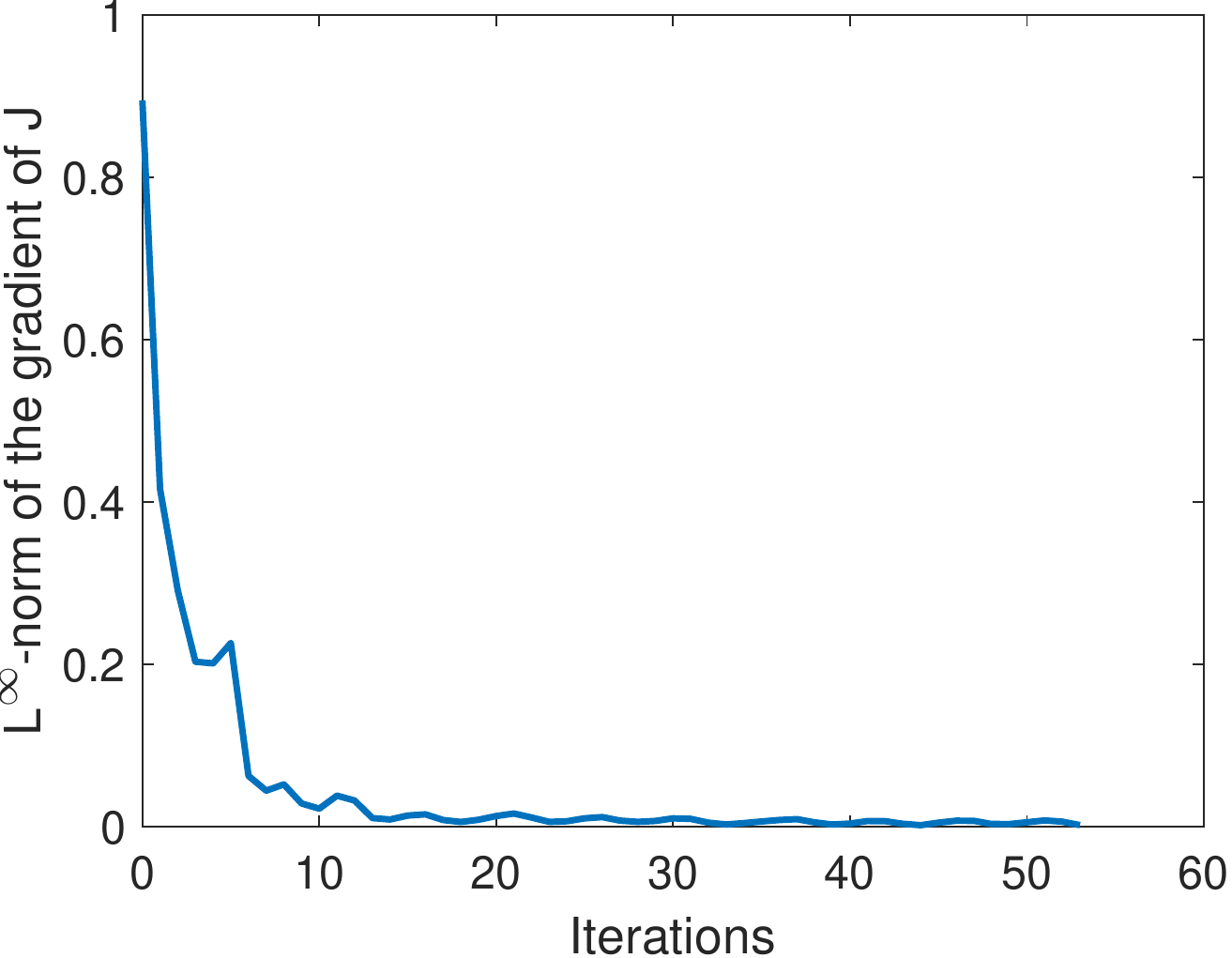}
\end{tabular}
 \caption{Simulation results for Example 3: History of the cost function $J$ and the $L^\infty$-norm of $J^\prime$ ($\epsilon=0.03$ and $\rho=0.0001$).}
  \label{geom17}
  \end{center}
 \end{figure}

 
 {\color{black}
 \begin{remark}
 The accuracy of  the numerical results presented in this paper depend on the initialization of  the proposed algorithm
 and the choice  of the regularization parameter. Meta-heuristic algorithms can be adapted to select a  good initialization  and   avoid being  trapped in a local  minima.
 The optimal  choice of the regularization parameter is based on some selection methods (such as  the discrepancy principle, the generalized cross validation, the L-curve criterion) which is  beyond the scope of this paper.
 \end{remark}
 }
 
\section{Conclusion and Outlook} 
In this paper, we dealt with the identification of Lam\'e parameters in linear elasticity. We introduced the inverse problem and the corresponding
Neumann-to-Dirichlet operator. Based on this, we analyzed the connection between the Lam\'e parameters and the Neumann-to-Dirichlet operator which led to a monotonicity
result. In order to prove a Lipschitz stability estimate {\color{black} for Lam\' e parameters which belong to a known finite subspace with a priori known bounds as well as certain regularity and monotonicity properties}, we applied the monotonicity result combined with the localized potentials. The numerical solution of the
inverse problem itself, was obtained via the minimization of a Kohn-Vogelius-type cost functional. In more detail, the reconstruction was performed via an iterative algorithm based on a 
quasi-Newton method. Finally, we presented our numerical examples and discussed them. 
{\color{black} We want to remark that the monotonicity properties of the Neumann-to-Drichlet operator as well 
as the results of the localized potentials build the basis for the monotonicty methods for linear elasticity (see \cite{monotonicity}) .}
 \\

\end{document}